\documentclass[smallextended,envcountsect]{svjour3}
\smartqed
\usepackage{graphicx}
\journalname{JOTA}

\usepackage[utf8]{inputenc}
\usepackage[english]{babel}
\usepackage{amsmath}
\usepackage{amsfonts}
\usepackage{amssymb}
\usepackage{mathrsfs}  
\usepackage{mathtools} 
\usepackage{enumerate}
\usepackage{xcolor}
\usepackage{tikz}
\usepackage{cite}
\usepackage{enumitem,hyperref,graphicx}
\usepackage{nameref}

\newenvironment{claiminproof}[2]{\par\noindent\emph{ Claim #1: }\space#2}{ }
\newenvironment{claimproof}[2]{\par\noindent\emph{ Proof of Claim #1: }\space#2}{\hfill $\square$}

\makeatletter
\def\namedlabel#1#2{\begingroup
    #2%
    \def\@currentlabel{#2}%
    \phantomsection\label{#1}\endgroup}
\makeatother

\newcommand{\B}{\mathbb{B}}

\newcommand{\R}{\mathbb{R}}

\begin{document}

\title{A Fixed-Point Approach to History-Dependent Sweeping Processes}
\titlerunning{A Fixed-Point Approach to History-Dependent Sweeping Processes}


\author{Mat\'ias Godoy, Manuel Torres-Valdebenito and  Emilio Vilches}

\institute{Mat\'ias Godoy \at 
            Facultad de Ingenier\'ia, Arquitectura y Dise\~no,
            Universidad San Sebasti\'an,
            Santiago, Chile\\
            matias.godoy@uss.cl
            \and
            Manuel Torres-Valdebenito \at
             Departamento de Ingenier\'ia Matem\'atica, Universidad de Chile,
             Santiago,  Chile\\
            manuel.torres@ug.uchile.cl
           \and
            Emilio Vilches \at Instituto de Ciencias de la Ingenier\'ia, 
               Universidad de O'Higgins, 
               Rancagua, Chile\\
               emilio.vilches@uoh.cl
}

\date{Received: date / Accepted: date}

\maketitle

\begin{abstract} 
   In this paper, we study the well-posedness of state-dependent and state-independent sweeping processes driven by prox-regular sets and perturbed by a history-dependent operator. Our approach, based on an enhanced version of Gronwall's lemma and fixed-point arguments, provides an efficient framework for analyzing sweeping processes. In particular, our findings recover all existing results for the class of Volterra sweeping processes and provide new insights into history-dependent sweeping processes. Finally, we apply our theoretical results to establish the well-posedness of a viscoelastic model with long memory.
\end{abstract}



\section{Introduction}

The sweeping process is a differential inclusion involving normal cones to a family of moving sets. Since its introduction by J.-J. Moreau in a series of papers \cite{MR0637727, MR0637728, MR0508661}, it has become a natural mathematical modeling tool for problems with constraints, as well as for modeling various phenomena in contact mechanics, electrical circuits, crowd motion problems, among others. We refer to \cite{MR3753582, MR3908332, MR4403784} for further details. 

The sweeping process, originally studied by Moreau for convex sets, was later extended to the context of non-convex sets by various authors. Notably, the seminal works of Lionel Thibault \cite{MR1994056,MR2399209} developed the well-posedness theory for a family of uniformly prox-regular sets, which constitute a natural and broad framework for sweeping processes. In parallel, the perturbed and state-dependent cases were developed, generalizing classical results in differential equations and paving the way for further advancements in mathematical modeling. See, for example, \cite{MR2159846, MR2179241, MR3222899,MR2328857,MR3626639,MR3813128,MR4421900}.

Recently, a new variant of the sweeping process was proposed in \cite{MR4099068}. This variant, now known as the Volterra Sweeping Process, incorporates an integral term into the classical perturbed sweeping process. This integral term facilitates the generalization of Volterra differential equations and enables the modeling of constrained processes where the velocity depends on the trajectory at previous times. For well-posedness results, we refer to \cite{MR4492538, MR4422386, Vilches-2024}.

In this paper, we generalize all the aforementioned developments by considering a history-dependent operator, which accounts for constrained problems where the velocity depends on its history but not necessarily as an integral form. Our approach, which is based on an enhanced version of Gronwall’s lemma proved in \cite{Vilches-2024} and fixed-point arguments, provides an efficient framework for analyzing state-dependent and state-independent Volterra sweeping processes. 

It is worth emphasizing that history-dependent operators have been used to model viscoelastic materials \cite{MR4403784,MR4382708,MR3752610,MR2976197}, as in these materials, the displacement field depends on the trajectory. Many materials can be considered viscoelastic due to their properties, including soft and hard tissues (e.g., skin, cartilage, bone), synthetic polymers, elastomers, and others. For more details, we refer to \cite{lakes2009viscoelastic,christensen2012theory}. We apply our theoretical results to prove the well-posedness of a problem in contact mechanics with long memory.

The paper is organized as follows. After presenting some mathematical preliminaries, in Section \ref{section:3}, we summarize the main assumptions used throughout the paper. In Section \ref{section_4}, we establish the existence and uniqueness of solutions for state-independent sweeping processes driven by prox-regular sets and perturbed by history-dependent operators. Subsequently, in Section \ref{section-5}, we demonstrate the existence of solutions for state-dependent sweeping processes driven by prox-regular sets and perturbed by history-dependent operators. Then, in Section \ref{section-6}, as a consequence of our results, we establish the existence of solutions for Volterra sweeping processes. The paper concludes with an application to viscoelastic models with long memory.

\section{Preliminaries}

Let $(\mathcal{H},\langle \cdot,\cdot\rangle)$ be a real Hilbert space.  As usual, the norm of $\mathcal{H}$ is defined as $ \Vert x\Vert:=\sqrt{\langle x,x\rangle}$ and the closed unit ball is denoted as $\mathbb{B}$. The real numbers will be denoted by $\mathbb{R}$ and for $T>0$, we set $I:=[0,T]$. 
Let $S\subset \mathcal{H}$ be a nonempty closed set. The distance from a set $S$ to a point $x\in \mathcal{H}$ is defined as $d_{S}(x):=\inf_{y\in S}\Vert x-y\Vert$.  The set of points where the distance from $S$ to $x\in \mathcal{H}$ is attained is denoted by $\operatorname{Proj}_{S}(x)$. This set is possibly empty and, when it consists of a single point, is denoted by $\operatorname{proj}_S(x)$.  \\
Given two closed sets $A,B\subset\mathcal{H}$, the excess of $A$ over $B$ is defined as  $$\operatorname{exc}(A;B) := \sup_{x\in A} d_{B}(x).$$ Moreover, it is well-known that $\operatorname{exc}(A;B) = \inf\{\varepsilon>0: A\subset B+\varepsilon\B\}$.  Additionally, the Hausdorff distance between $A$ and $B$ is defined as
\begin{equation*}
    \operatorname{Haus}(A,B):=\max\{\operatorname{exc}(A;B),\operatorname{exc}(B;A)\},
\end{equation*}
which can be characterized through the formula (see, e.g., \cite[Lemma~3.74]{MR2378491}):
\begin{eqnarray*}
  \operatorname{Haus}(A,B)= \sup_{x\in \mathcal{H}}\vert d_A(x)-d_B(x)\vert.
\end{eqnarray*}
We refer to \cite{MR2378491} for more details.
\subsection{Convex and Variational Analysis Tools}
Given a function $f\colon\mathcal{H}\to\R\cup\{+\infty\}$, we say that $f\in \Gamma_0(\mathcal{H})$ if $f$ is proper, convex, and lower semicontinuous (l.s.c.). Given $f\in\Gamma_{0}(\mathcal{H})$, we say that $x^{*}$ is an element of the Fénchel subdifferential of $f$ at $x\in\mathcal{H}$, denoted by $\partial f(x)$, if 
$$
    f(y) \geq f(x) + \langle x^{*},y-x\rangle \quad\textrm{for all }y\in\mathcal{H}. 
$$
The Legendre-Fenchel conjugate $f^{*}\colon\mathcal{H}\to \R\cup\{+\infty\}$ of a function $ f\in \Gamma_0(\mathcal{H})$ is defined by
$$
    f^{*}(x^{*}) = \sup_{x\in\mathcal{H}}\{\langle x^{*},x\rangle - f(x)\} \quad\textrm{ for all }x^{*}\in\mathcal{H}.
$$
Moreover, it is well-known that the following equivalences hold:
$$
    x^{*} \in\partial f(x) \iff x \in\partial f^{*}(x^{*}) \iff \langle x^{*},x\rangle = f(x) + f^{*}(x^{*}).
$$
Besides, given a non-empty, closed, and convex set $ C \subset \mathcal{H} $, the indicator function of $C$ is defined as $\iota_{C} \colon \mathcal{H} \to \mathbb{R} \cup \{ +\infty \} $, given by
\begin{eqnarray*}
    \iota_{C}(x) = 
    \begin{cases}
        0       & \text{if } x \in C; \\ 
        +\infty & \text{if } x \notin C,
    \end{cases}
\end{eqnarray*}
and the support function of $C$ is defined as $\sigma_{C} \colon \mathcal{H} \to \mathbb{R} \cup \{ +\infty \}$, given by
\begin{eqnarray*}
    \sigma_{C}(x^{*}) = \sup_{x \in C} \langle x^{*}, x \rangle \quad \text{for all } x^{*} \in \mathcal{H}.
\end{eqnarray*}
Moreover, the following relations hold $\sigma_{C} = (\iota_{C})^{*}$ and $\iota_{C} = (\sigma_{C})^{*}$.

Given a closed and nonempty set $S\subset \mathcal{H}$ and $x\in S$. We say that $v$ belongs to the proximal normal cone $N^{P}(S;x)$ if there exists $\sigma \geq 0$ such that
$$
\langle v,y-x\rangle \leq \sigma\Vert y-x\Vert^2 \textrm{ for all } y\in S.
$$
Whenever $x\notin S$, we set $N^{P}(S;x)=\emptyset$. We refer to \cite{MR1488695} for more details about proximal calculus. 
Now, we recall the concept of a uniformly prox-regular set. Introduced by Federer in the finite-dimensional setting (see \cite{MR110078}) and later developed by Rockafellar, Poliquin, and Thibault in \cite{MR1694378},   prox-regularity generalizes and unifies the classes of convex sets and nonconvex bodies with $C^2$ boundary. \begin{definition}
    Let $S$ be a closed subset of $\mathcal{H}$ and $\rho\in ]0,+\infty]$. We say that $S$ is $\rho-$uniformly prox-regular provided that, for all $x\in S$ and all $v\in N^P(S;x)\cap \mathbb{B}$ one has 
    \begin{equation*}
       x\in \operatorname{Proj}_S(x+tv) \textrm{ for any } t\leq \rho.
    \end{equation*}
\end{definition}
It is important to emphasize that convex sets are $\rho-$uniformly prox-regular for any $\rho>0$. 
The following result summarizes the main characterization of uniform prox-regularity. We refer to \cite{MR2768810,MR4659163} for further results.	
\begin{proposition}
    \label{prop:prox}
    Let $S$ be a nonempty closed subset of $\mathcal{H}$ and $\rho\in ]0,+\infty]$. The following assertions are equivalent:
    \begin{enumerate}
        \item[(i)] The set $S$ is $\rho-$uniformly prox-regular.
        \item[(ii)] For all $x, x^{\prime}\in S$ and $\zeta \in N^{P}(S;x)$ one has
        \begin{equation*}
            \langle \zeta , x^{\prime} - x \rangle \leq \frac{1}{2\rho} \|\zeta \| \|x^{\prime}-x\|^{2}.
        \end{equation*}
        \item[(iii)] For any $x_{i}\in S$, $\zeta_{i}\in N^{P}(S;x_{i})\cap \mathbb{B}$ with $i=1,2$ one has
        \begin{eqnarray*}
            \langle \zeta_{1}-\zeta_{2} , x_{1}-x_{2} \rangle \geq - \frac{1}{\rho} \|x_{1}-x_{2}\|^{2}.
        \end{eqnarray*}
        that is, the set-valued mapping $x\mapsto N^{P}(S;x)\cap\B$ is $\frac{1}{\rho}$-hypomonotone.
        \item[(iv)] For any positive $\gamma<1$ the projection mapping $\operatorname{proj}_S$ is well-defined and Lipschitz continuous on $U^{\gamma}_{\rho}(S):=\{z\in \mathcal{H}: d_S(x)<\gamma \rho\}$ with $1/(1-\gamma)$ as a Lipschitz constant, i.e., 
        \begin{equation*}
            \Vert \operatorname{proj}_S(x_{1})-\operatorname{proj}_S(x_{2})\Vert \leq \frac{1}{1-\gamma}\|x_{1}-x_{2}\| \quad \textrm{ for all } x_1, x_2\in U^{\gamma}_{\rho}(S).
        \end{equation*}
    \end{enumerate}
\end{proposition}
\subsection{Nonlinear Analysis Tools}
Let $A$ be a bounded subset of $\mathcal{H}$. We define the \textit{Kuratokwki measure of non-compactness} of $A$, $\alpha(A)$, as
\begin{equation*}
    \alpha(A) := \inf\big\{d>0\ :\ \text{$A$ admits a finite cover by sets of diameter $\leq d$}\big\},
\end{equation*}
and the \textit{Hausdorff measure of non-compactness} of $A$, $\beta(A)$, as
\begin{equation*}
    \beta(A) := \inf\big\{r>0\ :\ \text{$A$ can be covered by finitely many ball of radius $r$}\big\}.
\end{equation*}
In \textit{Hilbert spaces} the relation between these two concepts is given by the inequality for $A\subset\mathcal{H}$ bounded:
$$
    \sqrt{2}\beta(A)\leq \alpha(A)\leq 2\beta(A).
$$
The following proposition gathers the main properties of \textit{Kuratowski and Hausdorff measures of non-compactness} (see \cite[Proposition 7.2]{MR0787404}, \cite[Section 9.2, Proposition 9.1]{MR1189795}).
\begin{proposition}     
    \label{prop:PropertiesAboutNonCompactnessMeasures}
    Let $\mathcal{H}$ be an infinite dimensional Hilbert space and $B$, $B_{1}$, $B_{2}$ be bounded subsets of $\mathcal{H}$. Let $\gamma$ be either the \textit{Kuratowski or the Hausdorff measures of non-compactness}. Then:
    \begin{enumerate}
        \item[(a)] $\gamma(B) = 0$ if and only if $\overline{B}$ is compact.
        \item[(a)] $\gamma(\lambda B) = |\lambda|\gamma(B)$ for every $\lambda\in\R$.
        \item[(b)] $\gamma(B_{1}+B_{2}) \leq \gamma(B_{1})+\gamma(B_{2})$.
        \item[(c)] $B_{1}\subset B_{2}$ implies $\gamma(B_{1}) \leq \gamma(B_{2})$.
        \item[(d)] $\gamma(\operatorname{conv}B) = \gamma(B)$.
        \item[(e)] $\gamma(\overline{B}) = \gamma(B)$.
    \end{enumerate}
\end{proposition}  

We will denote by $C(I;\mathcal{H})$ the set of all \textit{continuous functions} from $I$ to $\mathcal{H}$. The norm of the \textit{uniform convergence} on $C(I;\mathcal{H})$ will be denoted by $\|\cdot\|_{\infty}$. We denote by $L^{1}([0,T];\mathcal{H})$ the space of \textit{$\mathcal{H}$-valued Lebesgue integrable functions} defined over the interval $[0,T]$. We say that $x\in\operatorname{AC}([0,T];\mathcal{H})$ if there exists $f\in L^{1}([0,T];\mathcal{H})$ and $x_{0}\in\mathcal{H}$ such that
\begin{eqnarray*}
    x(t) = x_{0} + \int_{0}^{t} f(s)ds\ \forall t\in[0,T].
\end{eqnarray*}

The following result, proved in \cite{Vilches-2024}, is an enhanced version of the classical Gronwall inequality.
\begin{lemma}[Enhanced Gronwall Inequality]\label{lem:enhanced}
   Let $I:=[0,T]$, and let $\rho\colon I \to \mathbb{R}$ be a nonnegative absolutely continuous function. Let $K_1, K_2, \varepsilon \colon I \to \mathbb{R}_+$, and  $K_3 \colon I\times I \to \mathbb{R}_+$ be nonnegative measurable functions such that
    $$
        t\mapsto K_1(t) \textrm{ and } t\mapsto K_2(t)\int_{0}^t K_3(t,s)ds \textrm{ are  integrable.}
    $$
    Suppose that
    $$
        \dot{\rho}(t)\leq \varepsilon(t)+K_1(t)\rho(t)+K_2(t)\int_{0}^t K_3(t,s)\rho(s)ds \quad \textrm{ for a.e. } t\in I.
    $$
    Then,  one has
    $$
    \rho(t)\leq \rho(0)\exp\left(\int_{0}^t \gamma(s)ds\right)+\int_{0}^t \varepsilon(s)\exp\left(\int_s^t \gamma(\tau)d\tau\right)ds \quad \textrm{ for all }t\in I,
    $$
    where  $\gamma(t):=K_1(t)+K_2(t)\int_{0}^t K_3(t,s)ds$.
\end{lemma}
The following result will be used in the proof of Theorem \ref{teo:Existence1}.
\begin{lemma}
    \label{lem:cancel}
    Let $\Theta\colon I \to \mathbb{R}$ be a nonnegative continuous function. Assume that $\alpha$ and $\beta$ are two nonnegative integrable functions such that 
    \begin{equation*}
        \Theta^{2}(t) \leq\int_{0}^{t} \alpha(s)\Theta(s)ds + \int_{0}^{t} \beta(s)\Theta^{2}(s)ds \quad \textrm{ for all } t\in I.
    \end{equation*}
    Then, for all $t\in I$, one has
    \begin{equation*}
        \Theta(t) \leq \int_0^t \exp\left(\int_s^t \beta(\tau)d\tau\right)\alpha(s)ds.
    \end{equation*}
\end{lemma}
\begin{proof}  
    We will prove that
    \begin{equation}\label{Gron}
        \Theta(t)\leq \int_0^t \alpha(\tau)d\tau+\int_0^t \beta(\tau)\Theta(\tau)d\tau \textrm{ for all } t\in I,
    \end{equation}
    which, by virtue of the classical Gronwall's inequality, will imply the result. \newline
    If $\Theta\equiv 0$ the result is obvious. Otherwise, assume that $\Theta \not\equiv 0$ and  define $t^{\ast}:=\inf \{s\in I \colon \Theta(s)>0\}$.   On the one hand, we observe that $\Theta(t)=0$ for any $t\in [0,t^{*}]$, hence, \eqref{Gron} holds for any $t\in [0,t^{*}]$. On the other hand, from hypothesis, it is clear that for all $t\in ]t^{\ast},T]$, one has
    \begin{equation*}
        \sup_{s\in [0,t]}\Theta(s)^2\leq \sup_{s\in [0,t]}\Theta(s)\left(\int_0^t \alpha(\tau)d\tau+\int_0^t \beta(\tau)\Theta(\tau)d\tau\right),
    \end{equation*}
    which, being $\sup_{s\in [0,t]}\Theta(s)>0$ for all $t\in ]t^{\ast},T]$, implies that for $t\in ]t^{\ast},T]$
    \begin{equation*}
        \Theta(t)\leq \int_0^t \alpha(\tau)d\tau+\int_0^t \beta(\tau)\Theta(\tau)d\tau,
    \end{equation*}
    which finishes the proof. 
\end{proof}
An operator $\mathcal{R}\colon C(I;\mathcal{H})\to C(I;\mathcal{H})$ is called an \textit{History-Dependent Operator}  if there exists $\kappa>0$, such that for all $x,y\in C(I;\mathcal{H})$ and $t\in I$, one has
\begin{equation*}
    \|\mathcal{R}(x)(t) - \mathcal{R}(y)(t)\| \leq \kappa\int_{0}^{t} \|x(s) - y(s)\| ds.
\end{equation*}
For more details, we refer to \cite[Chapter 2]{MR3752610}. The following result provides a fixed point result for history-dependent operators (see, e.g.,  \cite[Chapter~2]{MR3752610}).
\begin{proposition}\label{teo:fixed}
    Let $\mathcal{S}\colon C(I;E)\to C(I;F)$ be a history dependent operator, where $E,F$ are two Banach spaces.  Then, the operator $\mathcal{S}$ has a unique fixed point.
\end{proposition}

\section{Technical assumptions}\label{section:3}

For ease of presentation, in this section, we gather the hypotheses used along the  paper.
\begin{enumerate}
    \item[\namedlabel{Hf}{$(\mathcal{H}^f)$}] The function $f\colon I\times\mathcal{H}\to\mathcal{H}$ satisfies
    \begin{enumerate}
        \item[(i)] For each $x\in\mathcal{H}$, the map $t \mapsto f(t,x)$ is measurable. 
        \item[(ii)] There exists two nonnegative integrable functions $\alpha$ and $\beta$  such that 
        \begin{equation*}
            \Vert f(t,x)\Vert \leq \alpha(t) \Vert x\Vert + \beta (t) \quad\textrm{ for all } (t,x)\in I\times \mathcal{H}.
        \end{equation*}
        \item[(iii)] For each $r>0$, there exists a nonnegative constant $\kappa_{f}^{r}$ such that
        \begin{equation*}
            \|f(t,x) - f(t,y)\| \leq \kappa_{f}^{r}\|x-y\| \quad\textrm{ for all }x,y\in r\B \textrm{ and } t\in I.
        \end{equation*}
    \end{enumerate}
\end{enumerate}
\begin{enumerate}
    \item[\namedlabel{Hg}{$(\mathcal{H}^g)$}]  The function $g\colon I \times I\times \mathcal{H}\to \mathcal{H}$ satisfies
    \begin{enumerate}
        \item For each $x\in \mathcal{H}$, the map $(t,s) \mapsto g(t,s,x)$ is measurable.
        \item For all $r>0$, there exists an integrable function $\mu_{r}\colon I\to \mathbb{R}_+$ such that for all $(t,s)\in D$ 
        $$
        \Vert g(t,s,x)-g(t,s,y)\Vert \leq \mu_{r}(t)\Vert x-y\Vert \textrm{ for all } x,y\in r\mathbb{B}.
        $$
        Here $D:=\{(t,s)\in I\times I \colon s\leq t\}$. 
        \item There exists a nonnegative integrable function $\sigma\colon D\to \mathbb{R}$ such that 
        $$
        \Vert g(t,s,x)\Vert \leq \sigma(t,s)(1+\Vert x\Vert) \textrm{ for all } (t,s)\in D \textrm{ and } x\in \mathcal{H}. 
        $$
    \end{enumerate}
\end{enumerate}

\begin{enumerate}
    \item[\namedlabel{HR}{$(\mathcal{H}^{\mathcal{R}})$}] The operator $\mathcal{R}\colon C(I;\mathcal{H}) \to C(I;\mathcal{H})$ is history-dependent of constant $\kappa_{\mathcal{R}}$, that is, there exists  $\kappa_{\mathcal{R}}\geq 0$ such that  for all $x,y\in C(I;\mathcal{H})$ one has
    \begin{equation*}
        \|\mathcal{R}(x)(t) - \mathcal{R}(y)(t)\| \leq \kappa_{\mathcal{R}} \int_{0}^{t} \|x(s) - y(s)\| ds \textrm{ for all } t\in I.
    \end{equation*}
\end{enumerate}
\begin{enumerate}
    \item[\namedlabel{HC}{$(\mathcal{H}^{C})$}] The map $C\colon[0,T]\rightrightarrows\mathcal{H}$ has nonempty, closed and $\rho-$uniformly prox-regular values, for some $\rho>0$. Moreover, there exists an absolutely continuous function $v\colon[0,T]\to\R$ such that
    \begin{equation*}
        \operatorname{Haus}(C(t),C(s)) \leq |v(t) - v(s)| \quad\textrm{ for all }t,s\in I.
    \end{equation*}
\end{enumerate}
\begin{enumerate}
\item[\namedlabel{HCx}{$(\mathcal{H}^{C}_{x})$}] The map $C\colon[0,T]\times\mathcal{H}\rightrightarrows\mathcal{H}$ has nonempty, closed and $\rho-$uniformly prox-regular values, for some $\rho>0$. Moreover, the following conditions hold:
\begin{enumerate}
\item[(i)] There exists an  absolutely continuous function $v\colon[0,T]\to\R$ and $L\in[0,1[$, such that
    \begin{equation*}
        \operatorname{Haus}(C(t,x),C(s,y)) \leq |v(t) - v(s)| + L\|x-y\| \textrm{ for all }t,s\in I, x,y\in\mathcal{H}.
    \end{equation*}
    \item[(ii)] For any $A,B\subset\mathcal{H}$ bounded and $r>0$, the set $C(t,A)\cap B \cap r\mathbb{B}$ is relatively compact.
     \end{enumerate}
\end{enumerate}

\section{History-Dependent Sweeping Processes}\label{section_4}

In this section, we provide an existence and uniqueness result for the solutions of the history-dependent sweeping process:
\begin{equation}\label{eq:SP1}
    \left\{
    \begin{aligned}
                \dot{x}(t) &\in -N_{C(t)}(x(t)) + f(t,x(t)) + \mathcal{R}(x)(t) \quad\textrm{ for a.e. }t\in I, \\
                x(0)&= x_{0} \in C(0),
    \end{aligned}
    \right.
\end{equation}
where $C$, $f$ and $\mathcal{R}$ satisfy \ref{HC}, \ref{Hf} and \ref{HR}, respectively. Before providing the main result of this section,  we provide a basic existence result for the sweeping process with integrable perturbation. We refer to \cite[Proposition~1]{MR2179241} for the proof.
\begin{proposition}
    \label{prop:Existence} 
    Let $\mathcal{H}$ be a real Hilbert space, and suppose that $C(\cdot)$ satisfies \ref{HC}. Let $h\colon I\to \mathcal{H}$ be a single-valued integrable mapping. Then, for any $x_0\in C(0)$ there exists a unique absolutely continuous solution $x(\cdot)$ for the following differential inclusion:
    \begin{equation*}
        \left\{
        \begin{aligned}
            \dot{x}(t)&\in -N_{C(t)}(x(t))+h(t)& \textrm{ a.e. } t\in I,\\
            x(0)&=x_0.
        \end{aligned}
        \right.
    \end{equation*}
    Moreover, $x(\cdot)$ satisfies the following inequality:
    $$
        \Vert \dot{x}(t)-h(t)\Vert \leq \Vert h(t)\Vert +\vert \dot{v}(t)\vert \quad \textrm{ a.e. } t\in I.
    $$
\end{proposition}
Now, we present the main result of this section. Based on a fixed-point result for history-dependent operators (see Proposition \ref{teo:fixed}), we prove the well-posedness for the problem \eqref{eq:SP1}.
\begin{theorem}
    \label{teo:Existence1}
    Assume that  \ref{HC}, \ref{Hf} and \ref{HR} hold. Then, for any initial condition $x_0\in C(0)$ there exists a unique absolutely continuous solution $x(\cdot)$ for  the problem \eqref{eq:SP1}. Moreover, the following bound holds:
    \begin{equation*}
        \begin{aligned}
            \|x(t)\|  \leq r(t)  \textrm{ for all }t\in I \textrm{ and } \|\dot{x}(t)\|  \leq q(t) \textrm{ for a.e. } t\in I,
        \end{aligned}
    \end{equation*}
    where, $r(\cdot)$ and $q(\cdot)$ are the functions defined by {\small 
    \begin{equation}\label{eq_rt}
        \begin{aligned}
            r(t)&:=\Vert x_0\Vert \exp\left(2\int_0^t (\alpha(s)+\kappa_{\mathcal{R}})ds \right)+\int_0^t \varepsilon(s)\exp\left(2\int_s^t (\alpha(\tau)+\kappa_{\mathcal{R}})d\tau \right)ds,\\
            q(t)&:= |\dot{v}(t)| + 2\alpha(t)r(t) + 2\beta(t) + 2\kappa_{\mathcal{R}}\int_{0}^{t} r(s)ds + 2\|\mathcal{R}(0)(t)\|,
        \end{aligned}
    \end{equation}}
    and $\varepsilon(t):=\vert \dot{v}(t)\vert +2\Vert \mathcal{R}(0)(t)\Vert +2\beta(t)$.
\end{theorem}

\begin{proof} 
    Given $y\in C(I;\mathcal{H})$, let us consider the function 
    $$
        h_y(t):=f(t,\tilde{y}(t)) + \mathcal{R}(\tilde{y})(t), 
    $$
    where $\tilde{y}(t) := \operatorname{proj}_{r(t)\mathbb{B}}(y(t))$. By virtue of \ref{Hf} and \ref{HR}, for $t\in I$, one has
    \begin{equation*}
        \begin{aligned}
        \Vert h_y(t)\Vert &\leq \alpha(t)\Vert \tilde{y}(t)\Vert +\beta(t)+\Vert \mathcal{R}(\tilde{y})(t)-\mathcal{R}(0)(t)\Vert +\Vert \mathcal{R}(0)(t)\Vert \\
        &\leq  \alpha(t) r(t) +\beta(t)+\kappa_{\mathcal{R}}\int_0^t r(s) ds +\Vert \mathcal{R}(0)(t)\Vert,
        \end{aligned} 
    \end{equation*}
    which shows that the map $t\mapsto h_y(t)$ is integrable. Hence, according to Proposition \ref{prop:Existence}, the differential inclusion:     
    \begin{equation}
        \label{eq:relevo}
        \left\{
        \begin{aligned}
            \dot{x}(t) &\in -N_{C(t)}(x(t)) + h_y(t)  & \textrm{ for a.e. }t\in I,\\
            x(0)&= x_{0} \in C(0),
        \end{aligned}
        \right.
    \end{equation}
    admits  a unique solution $x(\cdot)$, which is absolutely continuous and 
    $$
        \Vert \dot{x}(t)-h_y(t)\Vert \leq \vert \dot{v}(t)\vert +\Vert h_y(t)\Vert \quad  \textrm{ a.e. } t\in I,
    $$ 
    where $r(\cdot)$ is defined in formula \eqref{eq_rt}. Therefore, for a.e. $t\in I$, one has
    $$
        \Vert \dot{x}(t)-h_y(t)\Vert \leq \psi(t):=\vert \dot{v}(t)\vert+\alpha(t)r(t)+\beta(t)+\kappa_{\mathcal{R}}\int_0^t r(s)ds+\Vert \mathcal{R}(0)(t)\Vert.
    $$   
    Hence, the operator $\mathcal{F} \colon C(I;\mathcal{H}) \to C(I;\mathcal{H})$ defined as $\mathcal{F}(y):=x_y$ is a well-defined single-valued function.   We proceed to prove that the operator $\mathcal{F}$ is history-dependent. Indeed, let $y_1, y_2\in C(I;\mathcal{H})$ and set $x_i=\mathcal{F}(y_{i})$ for $i=1,2$. Hence, according to \eqref{eq:relevo}, for a.e. $t\in I$, 
    \begin{equation*}
        \frac{ \dot{x}_{i}(t) - h_{y_{i}}(t) }{ \psi(t) } \in -N_{C(t)}(x_{i}(t)) \cap  \B  \quad\textrm{ for } i=1,2,
    \end{equation*}
    which, by the $\rho-$uniformly prox-regularity of the sets $C(t)$ and assertion $(iii)$ of Proposition \ref{prop:prox}, implies that for a.e.  $t\in I$, one has
    \begin{equation*}
        \langle \dot{x}_{1}(t) - \dot{x}_{2}(t) -[h_{y_{1}}(t) - h_{y_{2}}(t)], x_{1}(t) - x_{2}(t) \rangle \leq \frac{\psi(t)}{\rho} \|x_{1}(t) - x_{2}(t)\|^{2}.
    \end{equation*}
    Therefore,  for a.e. $t\in I$, 
    \begin{equation*}
        \frac{1}{2}\frac{d}{dt}\Vert x_1(t)-x_2(t)\Vert^2\leq  \| h_{y_{1}}(t) - h_{y_{2}}(t) \| \cdot \| x_{1}(t) - x_{2}(t) \| + \frac{\psi(t)}{\rho} \| x_{1}(t) - x_{2}(t) \|^{2}.
    \end{equation*}
    Then, by integrating the above inequality and using Lemma \ref{lem:cancel}, we obtain that for all $t\in I$, one has
    \begin{equation}
        \label{cota-corta}
        \Vert x_1(t)-x_2(t)\Vert \leq 2\int_{0}^t \exp\left(2\int_s^t \frac{\psi(\tau)}{\rho}d\tau\right) \Vert h_{y_{1}}(s) - h_{y_{2}}(s)\Vert ds.
    \end{equation}
    Now, we proceed to estimate the term  $ \Vert h_{y_{1}}(t) - h_{y_{2}}(t)\Vert $. Indeed, by virtue of \ref{Hf} and \ref{HR}, we obtain that for a.e. $t\in I$, one has
    \begin{equation*}
        \begin{aligned}
            \|h_{y_{1}}(t) - h_{y_{2}}(t)\| 
            & \leq 
            \| f(t,\tilde{y}_{1}(t)) - f(t,\tilde{y}_{1}(t))\| + \|\mathcal{R}(\tilde{y}_{1})(t) - \mathcal{R}(\tilde{y}_{2})(t)\| \\
            & \leq 
            \kappa_{f}^{r(T)} \|\tilde{y}_{1}(t)-\tilde{y}_{2}(t)\| + \kappa_{\mathcal{R}}\int_{0}^{t} \|\tilde{y}_{1}(t) - \tilde{y}_{2}(t)\|ds \\
            & \leq 
            \kappa_{f}^{r(T)} \|y_{1}(t) - y_{2}(t)\|+ \kappa_{\mathcal{R}}\int_{0}^{t} \|y_{1}(s) - y_{2}(s)\|ds,
        \end{aligned}
    \end{equation*}
    where we have used that the map $x\mapsto \operatorname{proj}_{r(T)\mathbb{B}}(x)$ is Lipschitz of constant 1. \newline  Hence, from the previous calculation and inequality \eqref{cota-corta}, we deduce that for all $t\in I$, one has
    \begin{equation*}
        \begin{aligned}
            \Vert x_1(t)-x_2(t)\Vert \leq\, & 2\kappa_{f}^{r(T)}\int_{0}^t \exp\left(2\int_s^t \frac{\psi(\tau)}{\rho}d\tau\right) \Vert y_{1}(s) - y_{2}(s)\Vert ds\\
            &+2\kappa_{\mathcal{R}}\int_{0}^t \exp\left(2\int_s^t \frac{\psi(\tau)}{\rho}d\tau\right) \int_0^s  \Vert y_{1}(\tau) - y_{2}(\tau)\Vert d\tau ds,
        \end{aligned}
    \end{equation*}
    from which we deduce that $\mathcal{F}$ is a state-dependent operator. Therefore, by virtue of Proposition \ref{teo:fixed}, the history-dependent operator $\mathcal{F}$ has a unique fixed point $x(\cdot)$, which solves the differential inclusion:
    \begin{equation}
        \label{eq:relevo2}
        \left\{
        \begin{aligned}
            \dot{x}(t) &\in -N_{C(t)}(x(t)) +h_x(t)  & \textrm{ for a.e. }t\in I,\\
            x(0)&= x_{0} \in C(0),
        \end{aligned}
        \right.
    \end{equation}
    where $h_x(t):=  f(t,\tilde{x}(t))+\mathcal{R}(\tilde{x}(t))$ and $\tilde{x}(t):=\operatorname{proj}_{r(t)\mathbb{B}}(x(t))$. \newline  To finish the proof, it remains to show that
    $$
    \|x(t)\| \leq r(t) \textrm{ for all } t\in I \textrm{ and } \|\dot{x}(t)\| \leq q(t) \textrm{ for a.e. } t\in I.
    $$
    Indeed, by virtue of \ref{Hf} and \ref{HR}, for all $t\in I$, one has
    \begin{equation*}
        \begin{aligned}
            \Vert x(t)\Vert &\leq \Vert x_0\Vert +\int_0^t \Vert \dot{x}(s)\Vert ds\\
            &\leq \Vert x_0\Vert +\int_0^t (\vert \dot{v}(s)\vert + \Vert h_{x}(s)\Vert)ds+\int_0^t \Vert h_{x}(s)\Vert ds\\
            &\leq \Vert x_0\Vert +\int_0^t \vert \dot{v}(s)\vert ds +2\int_0^t  \Vert f(s,\tilde{x}(s))\Vert  ds+2\int_0^t  \Vert \mathcal{R}(\tilde{x})(s)\Vert  ds\\
            &\leq \Vert x_0\Vert +\int_0^t \vert \dot{v}(s)\vert ds+2\int_0^t  \Vert  \mathcal{R}(0)(s)\Vert  ds+2\int_0^t  \Vert f(s,\tilde{x}(s))\Vert  ds\\&+2\int_0^t  \Vert \mathcal{R}(\tilde{x})(s)- \mathcal{R}(0)(s)\Vert  ds\\
            &\leq \Vert x_0\Vert +\int_0^t \vert \dot{v}(s)\vert ds+2\int_0^t  \Vert  \mathcal{R}(0)(s)\Vert  ds+2\int_0^t \beta(s)ds\\
            &+2\int_0^t \alpha(s)\Vert \tilde{x}(s)\Vert ds+2\kappa_{\mathcal{R}}\int_0^t  \int_0^s  \Vert \tilde{x}(\tau)\Vert d\tau  ds\\
            &\leq \Vert x_0\Vert +\int_0^t \vert \dot{v}(s)\vert ds+2\int_0^t  \Vert  \mathcal{R}(0)(s)\Vert  ds+2\int_0^t \beta(s)ds\\
            &+2\int_0^t \alpha(s)\Vert {x}(s)\Vert ds+2\kappa_{\mathcal{R}}\int_0^t  \int_0^s  \Vert {x}(\tau)\Vert d\tau  ds,
        \end{aligned}
    \end{equation*}
    where we have used that the map $x\mapsto \operatorname{proj}_{r(t)\mathbb{B}}(x)$ is Lipschitz of constant 1. Therefore, for all $t\in I$, one has
    \begin{equation*}
        \begin{aligned}
            \Vert x(t)\Vert \leq \Vert x_0\Vert +\int_0^t \varepsilon(s)ds+2\int_0^t \alpha(s)\Vert x(s)\Vert ds+2\kappa_{\mathcal{R}}\int_0^t \int_0^s \Vert x(\tau)\Vert d\tau ds,
        \end{aligned}
    \end{equation*}
    where $\varepsilon(t):=\vert \dot{v}(t)\vert +2\Vert \mathcal{R}(0)(t)\Vert +2\beta(t)$. Hence, by virtue of Lemma \ref{lem:enhanced}, we obtain that, for all $t\in I$, one has {\small
    \begin{equation*}
        \Vert x(t)\Vert \leq r(t):=\Vert x_0\Vert \exp\left(2\int_0^t (\alpha(s)+\kappa_{\mathcal{R}})ds \right)+\int_0^t \varepsilon(s)\exp\left(2\int_s^t (\alpha(\tau)+\kappa_{\mathcal{R}})d\tau \right)ds,
    \end{equation*}}
    which implies that $x(\cdot)$ solves the history-dependent sweeping process \eqref{eq:SP1}.  Finally,  similarly to the above calculations, for a.e. $t\in I$, one has
    \begin{equation*}
        \begin{aligned}
            \Vert \dot{x}(t)\Vert &\leq \Vert \dot{x}(t)-h_{x}(t)\Vert +\Vert h_{x}(t)\Vert \\
            &\leq \vert \dot{v}(t)\vert +2 \Vert h_{x}(t)\Vert \\
            &\leq \vert \dot{v}(t)\vert + 2\Vert \mathcal{R}(0)(s)\Vert +2\alpha(t)\Vert x(t)\Vert+2\beta(t)+2\kappa_{\mathcal{R}}\int_0^t \Vert x(s)\Vert ds\\
            &\leq q(t):=\vert \dot{v}(t)\vert + 2\Vert \mathcal{R}(0)(s)\Vert +2\alpha(t)r(t)+2\beta(t)+2\kappa_{\mathcal{R}}\int_0^t r(s) ds,
        \end{aligned}
    \end{equation*}
    which ends the proof.
\end{proof}

\section{State-Dependent Sweeping Process}\label{section-5}

In this section, we consider the following state-dependent history-dependent sweeping process:
\begin{equation}
    \label{StateSP}
    \left\{
    \begin{aligned}
                \dot{x}(t)  &\in -N_{C(t,x(t))}(x(t)) + f(t,x(t)) + \mathcal{R}(x)(t) \quad\textrm{ for a.e. }t\in[0,T], \\
                x(0)&= x_{0} \in C(0,x_0),
        \end{aligned}
    \right.
\end{equation}
where $C$, $f$ and $\mathcal{R}$ satisfy \ref{HCx}, \ref{Hf} and \ref{HR}, respectively.  By means of the Schauder Fixed Point Theorem, we prove the existence of solutions for the above differential inclusion.

\begin{theorem}\label{teo:Existence2}
    Assume that \ref{HCx}, \ref{Hf} and \ref{HR} hold. Then, for any $x_0\in C(0,x_0)$, there exists at least one absolutely continuous solution $x(\cdot)$ for the problem \eqref{StateSP}. Moreover, the following bound holds:
       \begin{equation*}
       \begin{aligned}
        \|x(t)\|        &\leq r(t)+L\int_0^t \psi(s)\exp\left( 2\int_s^t (\alpha(\tau)+\kappa_{\mathcal{R}})d\tau \right)ds &\textrm{ for all }t\in I,\\
        \|\dot{x}(t)\|  &\leq \psi(t)  &\textrm{ for a.e. }t\in I,
        \end{aligned}
    \end{equation*}
    where $r(t)$ is defined in formula \eqref{eq_rt} and 
    \begin{equation*}
        \begin{aligned}
            \psi(t):=\frac{q(t)}{1-L}+\frac{2L}{1-L}(\alpha(t)+\kappa_{\mathcal{R}})\int_0^t \exp\left(2\int_t^s (\alpha(\tau)+\kappa_{\mathcal{R}})d\tau\right)ds
        \end{aligned}
    \end{equation*}  
\end{theorem}   

\begin{proof} 
    We say that $y\in \mathcal{K}$ if there exists $f\in L^1(I;\mathcal{H})$ such that $\Vert f(t)\Vert \leq \psi(t)$ for a.e. $t\in I$ and 
    $$
        y(t):=x_0+\int_0^t f(s)ds \textrm{ for all } t\in I.
    $$
    It is clear that the set $\mathcal{K}$ seen as a subset of $C(I;\mathcal{H})$ is nonempty, closed and convex. We observe that for a given $y\in \mathcal{K}$ the set-valued map $t\mapsto C(t,y(t))$ satisfies assumption \ref{HC}. Indeed, according to \ref{HCx}, for any $t,s\in I$ with $s\leq t$, one has
    \begin{equation*}
        \begin{aligned}
            \operatorname{Haus}(C(t,y(t)),C(s,y(s)))
            &\leq \vert v(t)-v(s)\vert +L\Vert y(t)-y(s)\Vert \\
            &\leq \vert v(t)-v(s)\vert +L\int_s^t \psi(\tau) d\tau.
        \end{aligned}
    \end{equation*}
    Therefore, by virtue of Theorem \ref{teo:Existence1}, for any $y\in \mathcal{K}$, there exists a unique absolutely continuous solution $x(\cdot)$ to the following differential inclusion:
    \begin{equation}\label{StateSP2}
        \left\{
        \begin{aligned}
                \dot{x}(t)  &\in -N_{C(t,y(t))}(x(t)) + f(t,x(t)) + \mathcal{R}(x)(t) \quad\textrm{ for a.e. }t\in I, \\
                x(0)&= x_{0}.
            \end{aligned}
        \right.
    \end{equation}
    Moreover, the following bounds hold: {\small
    \begin{equation*}
        \begin{aligned}
            \Vert x(t)\Vert 
            &\leq r_0(t):=r(t)+L\int_0^t \psi(s)\exp\left(2\int_s^t (\alpha(\tau)+\kappa_{\mathcal{R}})d\tau\right)ds  & \textrm{ for all } t\in I,\\
            \Vert \dot{x}(t)\Vert 
            &\leq q_0(t):=q(t)+L\psi(t)+2\alpha(t)L\exp\left(2\int_0^t  (\alpha(\tau)+\kappa_{\mathcal{R}})d\tau\right)\nu(t)\\
            &+2\kappa_{\mathcal{R}}L\int_0^t\exp\left(2\int_0^t  (\alpha(\tau)+\kappa_{\mathcal{R}})d\tau\right)\nu(s)ds  & \textrm{ a.e. } t\in I,
        \end{aligned}
    \end{equation*}}
    where 
    $$
            \nu(t):=\int_0^t \psi(s)\exp\left(2\int_0^s  (\alpha(\tau)+\kappa_{\mathcal{R}})d\tau\right)ds.
    $$ 
    Hence, we can consider the operator $\mathcal{F}\colon \mathcal{K} \to C(I;\mathcal{H})$ which assigns to each $y\in \mathcal{K}$ the unique solution $x(\cdot):=\mathcal{F}(y)$ of the problem \eqref{StateSP2}. Moreover, for a.e. $t\in I$, one has
    \begin{equation*}
        \begin{aligned}
            \Vert \dot{x}(t)\Vert  
            &\leq q(t)+L\psi(t)+2\alpha(t)L\exp\left(2\int_0^t  (\alpha(\tau)+\kappa_{\mathcal{R}})d\tau\right)\nu(t)\\
            &+2\kappa_{\mathcal{R}}L\int_0^t\exp\left(2\int_0^t  (\alpha(\tau)+\kappa_{\mathcal{R}})d\tau\right)\nu(s)ds= \psi(t),
        \end{aligned}
    \end{equation*}
    which proves that the operator $\mathcal{F}$ takes values in $\mathcal{K}$. 
    \begin{claiminproof}{1}
        The operator $\mathcal{F}\colon \mathcal{K} \to \mathcal{K}$ is continuous. 
    \end{claiminproof}
    \begin{claimproof}{1}
        Let us consider a sequence $(y_{n})\subset \mathcal{K}$ converging to $y\in \mathcal{K}$. Set $x_n:=\mathcal{F}(y_n)$ and $x:=\mathcal{F}(y)$. By using \eqref{StateSP2}, we obtain that
        \begin{equation*}
            \begin{aligned}
                \dot{x}_n(t)&\in -N_{C(t,y_n(t))}(x_n(t))+h_n(t) & \textrm{ a.e. } t\in I,\\
                \dot{x}(t)&\in -N_{C(t,y(t))}(x(t))+h(t) & \textrm{ a.e. } t\in I,
            \end{aligned}
        \end{equation*}
        where $h_n(t):=f(t,x_n(t))+\mathcal{R}(x_n)(t)$ and $h(t):=f(t,x(t))+\mathcal{R}(x)(t)$. Hence, by virtue of Proposition \ref{prop:Existence} and assumptions \ref{Hf} and \ref{HR}, we obtain that, for a.e. $t\in I$, one has
        \begin{equation*}
            \begin{aligned}
                \Vert \dot{x}_n(t)-h_n(t)\Vert &\leq \vert \dot{v}(t)\vert +L\psi(t)+\Vert h_n(t)\Vert \\
                &\leq m(t):=\vert \dot{v}(t)\vert +L\psi(t)+\alpha(t)r_0(t)+\beta(t)\\&+\kappa_{\mathcal{R}}\int_0^t r_0(s)ds+\Vert \mathcal{R}(0)(t)\Vert,
            \end{aligned}
        \end{equation*}
        which implies that for all $n\in \mathbb{N}$, one has
        \begin{equation*}
            \begin{aligned}
                \frac{\dot{x}_n(t)-h_n(t)}{m(t)}& \in -N_{C(t,y_n(t))}(x_n(t))\cap \mathbb{B}  & \textrm{ a.e. } t\in I.
            \end{aligned}
        \end{equation*}
        Similarly, 
        \begin{equation*}
            \begin{aligned}
                \frac{\dot{x}(t)-h(t)}{m(t)}& \in -N_{C(t,y(t))}(x(t))\cap \mathbb{B}  & \textrm{ a.e. } t\in I.
            \end{aligned}
        \end{equation*}
        Hence, By virtue of Proposition \ref{prop:prox}, we get that for all $c\in C(t,y_n(t))$ 
        \begin{equation}\label{desigualdad1}
            \begin{aligned}
                \langle -\dot{x}_n(t)+h_n(t),c-x_n(t)\rangle &\leq \frac{m(t)}{2\rho}\Vert c-x_n(t)\Vert^2
            \end{aligned}
        \end{equation}
        Similarly, for all $c\in C(t,y(t))$
        \begin{equation}\label{desigualdad2}
             \langle -\dot{x}(t)+h(t),c-x(t)\rangle \leq \frac{m(t)}{2\rho}\Vert c-x(t)\Vert^2. 
        \end{equation}
        Next, according to \ref{HCx}, for all $t\in I$, we have
        \begin{equation*}
            \begin{aligned}
                C(t,y_n(t))\subset C(t,y(t))+L\Vert y(t)-y_n(t)\Vert \mathbb{B},\\
                C(t,y(t))\subset C(t,y_n(t))+L\Vert y(t)-y_n(t)\Vert \mathbb{B},
            \end{aligned}
        \end{equation*}
        which implies the existence of $d_n(t), \tilde{d}_n(t)\in \mathbb{B}$ such that
        \begin{equation*}
            \begin{aligned}
                x(t)-L\Vert y_n(t)-y(t)\Vert \tilde{d}_n(t)&\in C(t,y_n(t)),\\
                x_n(t)-L\Vert y_n(t)-y(t)\Vert d_n(t)&\in C(t,y(t)).
            \end{aligned}
        \end{equation*}
        Therefore, by using the above quantities in \eqref{desigualdad1} and \eqref{desigualdad2}, respectively,  we obtain that for a.e. $t\in I$ 
        \begin{equation*}
            \begin{aligned}
                \langle -\dot{x}_n(t)+h_n(t),x(t)-x_n(t)\rangle &\leq Lm(t)\Vert y_n(t)-y(t)\Vert +\frac{m(t)}{\rho}\Vert  x(t)-x_n(t)\Vert^2\\
                &+\frac{m(t)L^2}{\rho}\Vert  y(t)-y_n(t)\Vert^2\\
                \langle -\dot{x}(t)+h(t),x_n(t)-x(t)\rangle &\leq  Lm(t)\Vert y_n(t)-y(t)\Vert+ \frac{m(t)}{\rho}\Vert x_n(t)-x(t)\Vert^2\\
                &+\frac{m(t)L^2}{\rho}\Vert  y(t)-y_n(t)\Vert^2.
            \end{aligned}
        \end{equation*}
        Therefore, for a.e. $t\in I$,
        \begin{equation*}
            \begin{aligned}
                \frac{1}{2}\frac{d}{dt}\Vert x_n(t)-x(t)\Vert^2 &=\langle \dot{x}_n(t)-\dot{x}(t),x_n(t)-x(t)\rangle\\
                &\leq 2Lm(t)\Vert y_n(t)-y(t)\Vert+2\frac{m(t)L^2}{\rho}\Vert y_n(t)-y(t)\Vert^2\\
                &+2\frac{m(t)}{\rho}\Vert x_n(t)-x(t)\Vert^2+\langle h_n(t)-h(t),x_n(t)-x(t)\rangle.
            \end{aligned}
        \end{equation*}
        Moreover, by using \ref{Hf} and \ref{HR}, we obtain that
        {\small 
        \begin{equation*}
            \begin{aligned}
                \langle h_n(t)-h(t),x_n(t)-x(t)\rangle &\leq \Vert h_n(t)-h(t)\Vert \cdot \Vert x_n(t)-x(t)\Vert \\
                &\leq k_f^{R}\Vert x_n(t)-x(t)\Vert^2+\kappa_{\mathcal{R}}\Vert x_n(t)-x(t)\Vert \int_0^t \Vert x_n(s)-x(s)\Vert ds,
            \end{aligned}
        \end{equation*}}
        where $R:=\sup_{t\in I} r_0(t)$ and $k_f^{R}$ is the constant given by \ref{Hf}. Therefore, for all $t\in I$, one has
        {\small 
        \begin{equation*}
            \begin{aligned}
                \Vert x_n(t)-x(t)\Vert^2 &\leq 4L\int_0^t m(s)\Vert y_n(s)-y(s)\Vert ds +\frac{4L^2}{\rho}\int_0^t m(s)\Vert y_n(s)-y(s)\Vert^2ds\\
                &+\int_0^t \left(4\frac{m(s)}{\rho}+2\kappa_f^R\right)\Vert x_n(s)-x(s)\Vert^2 ds+\kappa_{\mathcal{R}}\left(\int_0^t \Vert x_n(s)-x(s)\Vert ds\right)^2.
            \end{aligned}
        \end{equation*}}
        Set $\sigma(t):=\limsup \Vert x_n(t)-x(t)\Vert$. By taking limit in the above inequality, we obtain that for all $t\in I$
        \begin{equation*}
            \begin{aligned}
                \sigma^2(t)&\leq \int_0^t \left(4\frac{m(s)}{\rho}+2\kappa_f^R\right)\sigma^2(s) ds+\kappa_{\mathcal{R}}\left(\int_0^t \sigma(s) ds\right)^2\\
                &\leq  \int_0^t \left(4\frac{m(s)}{\rho}+2\kappa_f^R\right)\sigma^2(s) ds+\kappa_{\mathcal{R}} t \int_0^t \sigma^2(s) ds,
            \end{aligned}
        \end{equation*}
        where we have used Holder's inequality. Finally, from the classical Gronwall's inequality, we conclude that $\sigma \equiv 0$, which proves the continuity of $\mathcal{F}$.
    \end{claimproof}
    To apply Schauder's Fixed Point Theorem, it remains to show that $\mathcal{F}(\mathcal{K})$ is relatively compact in $C(I;\mathcal{H})$.
    \begin{claiminproof}{2}
        The set $\mathcal{F}(\mathcal{K})$ relatively compact in $C(I;\mathcal{H})$.
    \end{claiminproof}
    \begin{claimproof}{2} 
        On the one hand, since $\mathcal{K}$ is bounded in $W^{1,1}(I;\mathcal{H})$ by definition. It is clear that $\mathcal{F}(\mathcal{K})$ is equicontinuous. On the other hand, fix $t\in I$ and define 
        \begin{equation*}
        \mathcal{K}(t):=\{ y(t)\colon y\in \mathcal{K}\} \textrm{ and }   \mathcal{F}(\mathcal{K})(t):=\{\mathcal{F}(y)(t)\colon y\in \mathcal{K}\}.
        \end{equation*}
        We can observe that $\mathcal{F}(\mathcal{K})(t)\subset C(t,\mathcal{K}(t))\cap r_0(t)\mathbb{B}$. Hence, by using the notion measure of non-compactness, we obtain that
        \begin{equation*}
        \gamma(\mathcal{F}(\mathcal{K})(t))\leq \gamma\left(C(t,\mathcal{K}(t))\cap r_0(t)\mathbb{B}\right)=0,
        \end{equation*}
        where we have used assumption \ref{HCx} and the monotony of $\gamma$. Therefore, the set $\mathcal{F}(\mathcal{K})(t)$ is relatively compact in $\mathcal{H}$. Finally, by virtue of the Arzel\`a-Ascoli Theorem, we conclude that the set  $\mathcal{F}(\mathcal{K})$ is relatively compact. 
    \end{claimproof} \newline
    Finally, by Schauder's Fixed Point Theorem, the operator $\mathcal{F}$ has a fixed point, which is a solution of \eqref{StateSP}.
\end{proof}

\section{Volterra Sweeping Processes}\label{section-6}
In this section, we show that our main results (Theorems \ref{teo:Existence1} and \ref{teo:Existence2}) enables us to obtain the existence of solutions for the Volterra Sweeping Process: 
\begin{equation}\label{Volterra1}
    \left\{
    \begin{aligned}
                \dot{x}(t) &\in -N_{C(t)}(x(t)) + f(t,x(t)) + \int_0^t g(t,s,x(s))ds \textrm{ for a.e. }t\in I, \\
                x(0)&= x_{0} \in C(0),
    \end{aligned}
    \right.
\end{equation}
and the State-Dependent Volterra Sweeping Process: {\small
\begin{equation}\label{Volterra2}
    \left\{
    \begin{aligned}
                \dot{x}(t) &\in -N_{C(t,x(t))}(x(t)) + f(t,x(t)) + \int_0^t g(t,s,x(s))ds \textrm{ for a.e. }t\in I, \\
                x(0)&= x_{0} \in C(0,x_0).
    \end{aligned}
    \right.
\end{equation}}
\noindent Here we assume that $f$ and $g$ satisfy  \ref{Hf} and \ref{Hg}, respectively.  Hence, to address the above dynamical systems, we can consider the operator 
$$
\mathcal{R}(x)(t):=\int_0^t f(t,s,x(s))ds.
$$
Unfortunately, this operator does not satisfy Assumption \ref{HR}; hence, we cannot directly apply Theorems \ref{teo:Existence1} and \ref{teo:Existence2}, respectively. Nevertheless, by using the reparametrization technique developed in \cite{Vilches-2024}, we can, without loss of generality, assume that $\mathcal{R}$ satisfies Assumption \ref{HR}. Hence, on the one hand, the existence result for the Volterra Sweeping Process is the following:
\begin{theorem}\label{teo:Volterra1}
    Assume that  \ref{HC}, \ref{Hf} and \ref{Hg} hold. Then, for any initial condition $x_0\in C(0)$ there exists a unique absolutely continuous solution $x(\cdot)$ for  the problem \eqref{Volterra1}. 
\end{theorem}
On the other hand, the existence result for the State-Dependent Volterra Sweeping Process is the following:
\begin{theorem}\label{teo:Volterra2}
    Assume that \ref{HCx}, \ref{Hf} and \ref{Hg} hold. Then, for any $x_0\in C(0,x_0)$, there exists at least one absolutely continuous solution $x(\cdot)$ for the problem \eqref{Volterra2}. 
\end{theorem}   
\begin{remark} On the one hand, we note that Theorem \ref{teo:Volterra1} is not new; in fact, it was previously proven using different methods in \cite{MR4492538,Vilches-2024}. Our contribution lies in demonstrating that it can be derived through a fixed-point argument. On the other hand, to the best of our knowledge, the result presented in \ref{teo:Volterra1} is novel.
\end{remark}

\section{An Application to Viscoelastic Models with Long Memory}\label{mechanics}

In this section, we illustrate one of our theoretical results (Theorem \ref{teo:Existence1}) through the modeling of a contact mechanical problem with long memory. 

Let \(\Omega \subset \mathbb{R}^{d}\), with $d\in \{2,3\}$, be an open, bounded and connected set with a Lipschitz boundary \(\Gamma := \partial\Omega\). We assume that \(\Gamma\) can be decomposed into mutually disjoint measurable sets $\Gamma_{D}$, $\Gamma_{N}$ and $\Gamma_{C}$, with \(\Gamma_{D}\) having a positive Hausdorff measure. We are interested in the situation where \(\Omega\) describes a viscoelastic material with long memory, which is clamped on \(\Gamma_{D}\), subjected to a normal traction on \(\Gamma_{N}\), and experiences dry friction on \(\Gamma_{C}\) (see Fig. \ref{fig:ConfigPhysical}).
\begin{figure}[h!]
    \centering

    \tikzset{every picture/.style={line width=0.75pt}} 
    
    \begin{tikzpicture}[x=0.75pt,y=0.75pt,yscale=-0.6,xscale=0.6]
        
        \draw  [color={rgb, 255:red, 128; green, 128; blue, 128 }  ,draw opacity=1 ][fill={rgb, 255:red, 208; green, 2; blue, 27 }  ,fill opacity=0.3 ] (57.83,81.98) .. controls (94.63,67.17) and (196.22,63.81) .. (230.08,86.23) .. controls (263.94,108.65) and (295.23,123.46) .. (230.08,175.11) .. controls (164.93,226.76) and (105.68,220.41) .. (64.45,175.11) .. controls (23.23,129.8) and (21.02,96.79) .. (57.83,81.98) -- cycle ;
        \draw  [color={rgb, 255:red, 74; green, 144; blue, 226 }  ,draw opacity=1 ][fill={rgb, 255:red, 74; green, 144; blue, 226 }  ,fill opacity=0.3 ] (104.1,203.53) .. controls (161.88,229.57) and (204.21,183.65) .. (241.02,194.02) .. controls (277.82,204.39) and (272.3,199.9) .. (309.11,244.34) .. controls (345.92,288.78) and (62.51,272.48) .. (25.7,228.04) .. controls (-11.1,183.6) and (46.31,177.48) .. (104.1,203.53) -- cycle ;
        \draw [color={rgb, 255:red, 74; green, 74; blue, 74 }  ,draw opacity=1 ]   (85,202) -- (96.1,190.9) ;
        \draw [color={rgb, 255:red, 74; green, 74; blue, 74 }  ,draw opacity=1 ]   (188.1,208.3) -- (182.1,197.9) ;
        \draw [color={rgb, 255:red, 74; green, 74; blue, 74 }  ,draw opacity=1 ]   (119.1,76.3) -- (117.1,65.9) ;
        \draw  [dash pattern={on 0.84pt off 2.51pt}]  (126.32,61.98) .. controls (151.86,59.88) and (173.1,62.5) .. (197.1,67.3) .. controls (222.1,72.3) and (233.1,76.3) .. (249.1,88.3) .. controls (265.1,100.3) and (274.1,105.3) .. (276.1,119.3) .. controls (278.1,133.3) and (273.1,140.3) .. (264.1,154.3) .. controls (255.41,167.81) and (220.66,194.33) .. (196.67,205.15) ;
        \draw [shift={(194.1,206.27)}, rotate = 337.38] [fill={rgb, 255:red, 0; green, 0; blue, 0 }  ][line width=0.08]  [draw opacity=0] (8.93,-4.29) -- (0,0) -- (8.93,4.29) -- cycle    ;
        \draw [shift={(123.1,62.27)}, rotate = 354.5] [fill={rgb, 255:red, 0; green, 0; blue, 0 }  ][line width=0.08]  [draw opacity=0] (8.93,-4.29) -- (0,0) -- (8.93,4.29) -- cycle    ;
        \draw  [dash pattern={on 0.84pt off 2.51pt}]  (76.43,195.57) .. controls (59.63,183.93) and (18.36,140.43) .. (24.1,107.3) .. controls (30,73.29) and (75.04,64.94) .. (106.25,64.3) ;
        \draw [shift={(109.1,64.27)}, rotate = 179.68] [fill={rgb, 255:red, 0; green, 0; blue, 0 }  ][line width=0.08]  [draw opacity=0] (8.93,-4.29) -- (0,0) -- (8.93,4.29) -- cycle    ;
        \draw [shift={(79.1,197.27)}, rotate = 209.58] [fill={rgb, 255:red, 0; green, 0; blue, 0 }  ][line width=0.08]  [draw opacity=0] (8.93,-4.29) -- (0,0) -- (8.93,4.29) -- cycle    ;
        \draw [color={rgb, 255:red, 0; green, 0; blue, 0 }  ,draw opacity=1 ] [dash pattern={on 0.84pt off 2.51pt}]  (91.86,206.74) .. controls (120.54,221.67) and (135.92,227.53) .. (180.34,211.32) ;
        \draw [shift={(183.1,210.3)}, rotate = 159.44] [fill={rgb, 255:red, 0; green, 0; blue, 0 }  ,fill opacity=1 ][line width=0.08]  [draw opacity=0] (8.93,-4.29) -- (0,0) -- (8.93,4.29) -- cycle    ;
        \draw [shift={(89.1,205.3)}, rotate = 27.66] [fill={rgb, 255:red, 0; green, 0; blue, 0 }  ,fill opacity=1 ][line width=0.08]  [draw opacity=0] (8.93,-4.29) -- (0,0) -- (8.93,4.29) -- cycle    ;
        \draw [color={rgb, 255:red, 0; green, 0; blue, 0 }  ,draw opacity=0.63 ]   (51.58,114.87) -- (46.82,127.23) ;
        \draw [shift={(46.1,129.1)}, rotate = 291.07] [color={rgb, 255:red, 0; green, 0; blue, 0 }  ,draw opacity=0.63 ][line width=0.75]    (10.93,-3.29) .. controls (6.95,-1.4) and (3.31,-0.3) .. (0,0) .. controls (3.31,0.3) and (6.95,1.4) .. (10.93,3.29)   ;
        \draw [color={rgb, 255:red, 0; green, 0; blue, 0 }  ,draw opacity=0.63 ]   (111,180) -- (116.89,195.97) ;
        \draw [shift={(117.58,197.85)}, rotate = 249.76] [color={rgb, 255:red, 0; green, 0; blue, 0 }  ,draw opacity=0.63 ][line width=0.75]    (10.93,-3.29) .. controls (6.95,-1.4) and (3.31,-0.3) .. (0,0) .. controls (3.31,0.3) and (6.95,1.4) .. (10.93,3.29)   ;
        \draw [color={rgb, 255:red, 0; green, 0; blue, 0 }  ,draw opacity=0.63 ]   (139.58,171.65) -- (146.31,187.26) ;
        \draw [shift={(147.1,189.1)}, rotate = 246.7] [color={rgb, 255:red, 0; green, 0; blue, 0 }  ,draw opacity=0.63 ][line width=0.75]    (10.93,-3.29) .. controls (6.95,-1.4) and (3.31,-0.3) .. (0,0) .. controls (3.31,0.3) and (6.95,1.4) .. (10.93,3.29)   ;
        \draw [color={rgb, 255:red, 0; green, 0; blue, 0 }  ,draw opacity=0.63 ]   (72.58,124.87) -- (71.29,138.11) ;
        \draw [shift={(71.1,140.1)}, rotate = 275.56] [color={rgb, 255:red, 0; green, 0; blue, 0 }  ,draw opacity=0.63 ][line width=0.75]    (10.93,-3.29) .. controls (6.95,-1.4) and (3.31,-0.3) .. (0,0) .. controls (3.31,0.3) and (6.95,1.4) .. (10.93,3.29)   ;
        \draw [color={rgb, 255:red, 0; green, 0; blue, 0 }  ,draw opacity=0.63 ]   (117.58,158.65) -- (123.36,173.24) ;
        \draw [shift={(124.1,175.1)}, rotate = 248.39] [color={rgb, 255:red, 0; green, 0; blue, 0 }  ,draw opacity=0.63 ][line width=0.75]    (10.93,-3.29) .. controls (6.95,-1.4) and (3.31,-0.3) .. (0,0) .. controls (3.31,0.3) and (6.95,1.4) .. (10.93,3.29)   ;
        \draw [color={rgb, 255:red, 0; green, 0; blue, 0 }  ,draw opacity=0.63 ]   (141.58,148.65) -- (149.98,161.01) ;
        \draw [shift={(151.1,162.67)}, rotate = 235.83] [color={rgb, 255:red, 0; green, 0; blue, 0 }  ,draw opacity=0.63 ][line width=0.75]    (10.93,-3.29) .. controls (6.95,-1.4) and (3.31,-0.3) .. (0,0) .. controls (3.31,0.3) and (6.95,1.4) .. (10.93,3.29)   ;
        \draw [color={rgb, 255:red, 0; green, 0; blue, 0 }  ,draw opacity=0.63 ]   (160.58,170.85) -- (174.12,183.49) ;
        \draw [shift={(175.58,184.85)}, rotate = 223.03] [color={rgb, 255:red, 0; green, 0; blue, 0 }  ,draw opacity=0.63 ][line width=0.75]    (10.93,-3.29) .. controls (6.95,-1.4) and (3.31,-0.3) .. (0,0) .. controls (3.31,0.3) and (6.95,1.4) .. (10.93,3.29)   ;
        \draw [color={rgb, 255:red, 0; green, 0; blue, 0 }  ,draw opacity=0.63 ]   (93.58,136.87) -- (89.65,150.74) ;
        \draw [shift={(89.1,152.67)}, rotate = 285.84] [color={rgb, 255:red, 0; green, 0; blue, 0 }  ,draw opacity=0.63 ][line width=0.75]    (10.93,-3.29) .. controls (6.95,-1.4) and (3.31,-0.3) .. (0,0) .. controls (3.31,0.3) and (6.95,1.4) .. (10.93,3.29)   ;
        \draw [color={rgb, 255:red, 0; green, 0; blue, 0 }  ,draw opacity=0.63 ]   (141.58,124.87) -- (149.78,134.17) ;
        \draw [shift={(151.1,135.67)}, rotate = 228.61] [color={rgb, 255:red, 0; green, 0; blue, 0 }  ,draw opacity=0.63 ][line width=0.75]    (10.93,-3.29) .. controls (6.95,-1.4) and (3.31,-0.3) .. (0,0) .. controls (3.31,0.3) and (6.95,1.4) .. (10.93,3.29)   ;
        \draw [color={rgb, 255:red, 0; green, 0; blue, 0 }  ,draw opacity=0.63 ]   (156.58,96.87) -- (170.27,102.86) ;
        \draw [shift={(172.1,103.67)}, rotate = 203.66] [color={rgb, 255:red, 0; green, 0; blue, 0 }  ,draw opacity=0.63 ][line width=0.75]    (10.93,-3.29) .. controls (6.95,-1.4) and (3.31,-0.3) .. (0,0) .. controls (3.31,0.3) and (6.95,1.4) .. (10.93,3.29)   ;
        \draw [color={rgb, 255:red, 0; green, 0; blue, 0 }  ,draw opacity=0.63 ]   (186,175) -- (200.53,186.43) ;
        \draw [shift={(202.1,187.67)}, rotate = 218.19] [color={rgb, 255:red, 0; green, 0; blue, 0 }  ,draw opacity=0.63 ][line width=0.75]    (10.93,-3.29) .. controls (6.95,-1.4) and (3.31,-0.3) .. (0,0) .. controls (3.31,0.3) and (6.95,1.4) .. (10.93,3.29)   ;
        \draw [color={rgb, 255:red, 0; green, 0; blue, 0 }  ,draw opacity=0.63 ]   (161.58,150.85) -- (174.55,161.4) ;
        \draw [shift={(176.1,162.67)}, rotate = 219.15] [color={rgb, 255:red, 0; green, 0; blue, 0 }  ,draw opacity=0.63 ][line width=0.75]    (10.93,-3.29) .. controls (6.95,-1.4) and (3.31,-0.3) .. (0,0) .. controls (3.31,0.3) and (6.95,1.4) .. (10.93,3.29)   ;
        \draw [color={rgb, 255:red, 0; green, 0; blue, 0 }  ,draw opacity=0.63 ]   (161.58,77.87) -- (177.17,82.14) ;
        \draw [shift={(179.1,82.67)}, rotate = 195.32] [color={rgb, 255:red, 0; green, 0; blue, 0 }  ,draw opacity=0.63 ][line width=0.75]    (10.93,-3.29) .. controls (6.95,-1.4) and (3.31,-0.3) .. (0,0) .. controls (3.31,0.3) and (6.95,1.4) .. (10.93,3.29)   ;
        \draw [color={rgb, 255:red, 0; green, 0; blue, 0 }  ,draw opacity=0.63 ]   (174.58,113.65) -- (187.35,120.7) ;
        \draw [shift={(189.1,121.67)}, rotate = 208.91] [color={rgb, 255:red, 0; green, 0; blue, 0 }  ,draw opacity=0.63 ][line width=0.75]    (10.93,-3.29) .. controls (6.95,-1.4) and (3.31,-0.3) .. (0,0) .. controls (3.31,0.3) and (6.95,1.4) .. (10.93,3.29)   ;
        \draw [color={rgb, 255:red, 0; green, 0; blue, 0 }  ,draw opacity=0.63 ]   (159.58,132.65) -- (171.55,142.4) ;
        \draw [shift={(173.1,143.67)}, rotate = 219.18] [color={rgb, 255:red, 0; green, 0; blue, 0 }  ,draw opacity=0.63 ][line width=0.75]    (10.93,-3.29) .. controls (6.95,-1.4) and (3.31,-0.3) .. (0,0) .. controls (3.31,0.3) and (6.95,1.4) .. (10.93,3.29)   ;
        \draw [color={rgb, 255:red, 0; green, 0; blue, 0 }  ,draw opacity=0.63 ]   (186.58,152.85) -- (200.44,162.09) ;
        \draw [shift={(202.1,163.2)}, rotate = 213.7] [color={rgb, 255:red, 0; green, 0; blue, 0 }  ,draw opacity=0.63 ][line width=0.75]    (10.93,-3.29) .. controls (6.95,-1.4) and (3.31,-0.3) .. (0,0) .. controls (3.31,0.3) and (6.95,1.4) .. (10.93,3.29)   ;
        \draw [color={rgb, 255:red, 0; green, 0; blue, 0 }  ,draw opacity=0.63 ]   (219,157) -- (235.21,162.55) ;
        \draw [shift={(237.1,163.2)}, rotate = 198.91] [color={rgb, 255:red, 0; green, 0; blue, 0 }  ,draw opacity=0.63 ][line width=0.75]    (10.93,-3.29) .. controls (6.95,-1.4) and (3.31,-0.3) .. (0,0) .. controls (3.31,0.3) and (6.95,1.4) .. (10.93,3.29)   ;
        \draw [color={rgb, 255:red, 0; green, 0; blue, 0 }  ,draw opacity=0.63 ]   (203,131) -- (219.34,139.72) ;
        \draw [shift={(221.1,140.67)}, rotate = 208.11] [color={rgb, 255:red, 0; green, 0; blue, 0 }  ,draw opacity=0.63 ][line width=0.75]    (10.93,-3.29) .. controls (6.95,-1.4) and (3.31,-0.3) .. (0,0) .. controls (3.31,0.3) and (6.95,1.4) .. (10.93,3.29)   ;
        \draw [color={rgb, 255:red, 0; green, 0; blue, 0 }  ,draw opacity=0.63 ]   (185.58,92.87) -- (200.18,97.11) ;
        \draw [shift={(202.1,97.67)}, rotate = 196.2] [color={rgb, 255:red, 0; green, 0; blue, 0 }  ,draw opacity=0.63 ][line width=0.75]    (10.93,-3.29) .. controls (6.95,-1.4) and (3.31,-0.3) .. (0,0) .. controls (3.31,0.3) and (6.95,1.4) .. (10.93,3.29)   ;
        \draw [color={rgb, 255:red, 0; green, 0; blue, 0 }  ,draw opacity=0.63 ]   (179,133) -- (195.34,141.72) ;
        \draw [shift={(197.1,142.67)}, rotate = 208.11] [color={rgb, 255:red, 0; green, 0; blue, 0 }  ,draw opacity=0.63 ][line width=0.75]    (10.93,-3.29) .. controls (6.95,-1.4) and (3.31,-0.3) .. (0,0) .. controls (3.31,0.3) and (6.95,1.4) .. (10.93,3.29)   ;
        \draw [color={rgb, 255:red, 0; green, 0; blue, 0 }  ,draw opacity=0.63 ]   (197,112) -- (213.34,120.72) ;
        \draw [shift={(215.1,121.67)}, rotate = 208.11] [color={rgb, 255:red, 0; green, 0; blue, 0 }  ,draw opacity=0.63 ][line width=0.75]    (10.93,-3.29) .. controls (6.95,-1.4) and (3.31,-0.3) .. (0,0) .. controls (3.31,0.3) and (6.95,1.4) .. (10.93,3.29)   ;
        \draw [color={rgb, 255:red, 0; green, 0; blue, 0 }  ,draw opacity=0.63 ]   (222,113) -- (238.34,121.72) ;
        \draw [shift={(240.1,122.67)}, rotate = 208.11] [color={rgb, 255:red, 0; green, 0; blue, 0 }  ,draw opacity=0.63 ][line width=0.75]    (10.93,-3.29) .. controls (6.95,-1.4) and (3.31,-0.3) .. (0,0) .. controls (3.31,0.3) and (6.95,1.4) .. (10.93,3.29)   ;
        \draw [color={rgb, 255:red, 0; green, 0; blue, 0 }  ,draw opacity=0.63 ]   (59.58,140.87) -- (60.92,155.21) ;
        \draw [shift={(61.1,157.2)}, rotate = 264.69] [color={rgb, 255:red, 0; green, 0; blue, 0 }  ,draw opacity=0.63 ][line width=0.75]    (10.93,-3.29) .. controls (6.95,-1.4) and (3.31,-0.3) .. (0,0) .. controls (3.31,0.3) and (6.95,1.4) .. (10.93,3.29)   ;
        \draw [color={rgb, 255:red, 0; green, 0; blue, 0 }  ,draw opacity=0.63 ]   (78.58,153.87) -- (79.94,170.21) ;
        \draw [shift={(80.1,172.2)}, rotate = 265.27] [color={rgb, 255:red, 0; green, 0; blue, 0 }  ,draw opacity=0.63 ][line width=0.75]    (10.93,-3.29) .. controls (6.95,-1.4) and (3.31,-0.3) .. (0,0) .. controls (3.31,0.3) and (6.95,1.4) .. (10.93,3.29)   ;
        \draw [color={rgb, 255:red, 0; green, 0; blue, 0 }  ,draw opacity=0.63 ]   (130.58,192.85) -- (137.3,208.36) ;
        \draw [shift={(138.1,210.2)}, rotate = 246.58] [color={rgb, 255:red, 0; green, 0; blue, 0 }  ,draw opacity=0.63 ][line width=0.75]    (10.93,-3.29) .. controls (6.95,-1.4) and (3.31,-0.3) .. (0,0) .. controls (3.31,0.3) and (6.95,1.4) .. (10.93,3.29)   ;
        \draw [color={rgb, 255:red, 0; green, 0; blue, 0 }  ,draw opacity=0.63 ]   (94.58,163.65) -- (95.94,180.21) ;
        \draw [shift={(96.1,182.2)}, rotate = 265.33] [color={rgb, 255:red, 0; green, 0; blue, 0 }  ,draw opacity=0.63 ][line width=0.75]    (10.93,-3.29) .. controls (6.95,-1.4) and (3.31,-0.3) .. (0,0) .. controls (3.31,0.3) and (6.95,1.4) .. (10.93,3.29)   ;
        \draw [color={rgb, 255:red, 0; green, 0; blue, 0 }  ,draw opacity=0.63 ]   (211,88) -- (227.34,96.72) ;
        \draw [shift={(229.1,97.67)}, rotate = 208.11] [color={rgb, 255:red, 0; green, 0; blue, 0 }  ,draw opacity=0.63 ][line width=0.75]    (10.93,-3.29) .. controls (6.95,-1.4) and (3.31,-0.3) .. (0,0) .. controls (3.31,0.3) and (6.95,1.4) .. (10.93,3.29)   ;
        \draw [color={rgb, 255:red, 0; green, 0; blue, 0 }  ,draw opacity=0.63 ]   (235,133) -- (251.34,141.72) ;
        \draw [shift={(253.1,142.67)}, rotate = 208.11] [color={rgb, 255:red, 0; green, 0; blue, 0 }  ,draw opacity=0.63 ][line width=0.75]    (10.93,-3.29) .. controls (6.95,-1.4) and (3.31,-0.3) .. (0,0) .. controls (3.31,0.3) and (6.95,1.4) .. (10.93,3.29)   ;
        \draw  [color={rgb, 255:red, 208; green, 2; blue, 27 }  ,draw opacity=1 ][fill={rgb, 255:red, 208; green, 2; blue, 27 }  ,fill opacity=0.3 ] (386.83,80.3) .. controls (423.63,65.48) and (525.22,62.13) .. (559.08,84.54) .. controls (592.94,106.96) and (624.23,121.78) .. (559.08,173.43) .. controls (493.93,225.08) and (434.68,218.73) .. (393.45,173.43) .. controls (352.23,128.12) and (350.02,95.11) .. (386.83,80.3) -- cycle ;
        \draw  [color={rgb, 255:red, 74; green, 144; blue, 226 }  ,draw opacity=1 ][fill={rgb, 255:red, 74; green, 144; blue, 226 }  ,fill opacity=0.3 ] (433.1,201.84) .. controls (490.88,227.89) and (533.21,181.97) .. (570.02,192.34) .. controls (606.82,202.71) and (601.3,198.21) .. (638.11,242.65) .. controls (674.92,287.09) and (391.51,270.8) .. (354.7,226.36) .. controls (317.9,181.92) and (375.31,175.8) .. (433.1,201.84) -- cycle ;
        \draw [color={rgb, 255:red, 74; green, 74; blue, 74 }  ,draw opacity=1 ]   (414,200.32) -- (425.1,189.22) ;
        \draw [color={rgb, 255:red, 74; green, 74; blue, 74 }  ,draw opacity=1 ]   (517.1,206.62) -- (511.1,196.22) ;
        \draw [color={rgb, 255:red, 74; green, 74; blue, 74 }  ,draw opacity=1 ]   (448.1,74.62) -- (446.1,64.22) ;
        \draw  [dash pattern={on 0.84pt off 2.51pt}]  (455.32,60.3) .. controls (480.86,58.2) and (502.1,60.82) .. (526.1,65.62) .. controls (551.1,70.62) and (562.1,74.62) .. (578.1,86.62) .. controls (594.1,98.62) and (603.1,103.62) .. (605.1,117.62) .. controls (607.1,131.62) and (602.1,138.62) .. (593.1,152.62) .. controls (584.41,166.13) and (549.66,192.64) .. (525.67,203.47) ;
        \draw [shift={(523.1,204.58)}, rotate = 337.38] [fill={rgb, 255:red, 0; green, 0; blue, 0 }  ][line width=0.08]  [draw opacity=0] (8.93,-4.29) -- (0,0) -- (8.93,4.29) -- cycle    ;
        \draw [shift={(452.1,60.58)}, rotate = 354.5] [fill={rgb, 255:red, 0; green, 0; blue, 0 }  ][line width=0.08]  [draw opacity=0] (8.93,-4.29) -- (0,0) -- (8.93,4.29) -- cycle    ;
        \draw  [dash pattern={on 0.84pt off 2.51pt}]  (405.43,193.88) .. controls (388.63,182.25) and (347.36,138.75) .. (353.1,105.62) .. controls (359,71.61) and (404.04,63.25) .. (435.25,62.62) ;
        \draw [shift={(438.1,62.58)}, rotate = 179.68] [fill={rgb, 255:red, 0; green, 0; blue, 0 }  ][line width=0.08]  [draw opacity=0] (8.93,-4.29) -- (0,0) -- (8.93,4.29) -- cycle    ;
        \draw [shift={(408.1,195.58)}, rotate = 209.58] [fill={rgb, 255:red, 0; green, 0; blue, 0 }  ][line width=0.08]  [draw opacity=0] (8.93,-4.29) -- (0,0) -- (8.93,4.29) -- cycle    ;
        \draw [color={rgb, 255:red, 0; green, 0; blue, 0 }  ,draw opacity=1 ] [dash pattern={on 0.84pt off 2.51pt}]  (420.86,205.06) .. controls (449.54,219.99) and (464.92,225.84) .. (509.34,209.64) ;
        \draw [shift={(512.1,208.62)}, rotate = 159.44] [fill={rgb, 255:red, 0; green, 0; blue, 0 }  ,fill opacity=1 ][line width=0.08]  [draw opacity=0] (8.93,-4.29) -- (0,0) -- (8.93,4.29) -- cycle    ;
        \draw [shift={(418.1,203.62)}, rotate = 27.66] [fill={rgb, 255:red, 0; green, 0; blue, 0 }  ,fill opacity=1 ][line width=0.08]  [draw opacity=0] (8.93,-4.29) -- (0,0) -- (8.93,4.29) -- cycle    ;
        \draw [color={rgb, 255:red, 0; green, 0; blue, 0 }  ,draw opacity=0.63 ]   (469.1,34.42) -- (469.1,55.42) ;
        \draw [shift={(469.1,57.42)}, rotate = 270] [color={rgb, 255:red, 0; green, 0; blue, 0 }  ,draw opacity=0.63 ][line width=0.75]    (10.93,-3.29) .. controls (6.95,-1.4) and (3.31,-0.3) .. (0,0) .. controls (3.31,0.3) and (6.95,1.4) .. (10.93,3.29)   ;
        \draw [color={rgb, 255:red, 0; green, 0; blue, 0 }  ,draw opacity=0.63 ]   (497.1,36.42) -- (495.27,57.42) ;
        \draw [shift={(495.1,59.42)}, rotate = 274.97] [color={rgb, 255:red, 0; green, 0; blue, 0 }  ,draw opacity=0.63 ][line width=0.75]    (10.93,-3.29) .. controls (6.95,-1.4) and (3.31,-0.3) .. (0,0) .. controls (3.31,0.3) and (6.95,1.4) .. (10.93,3.29)   ;
        \draw [color={rgb, 255:red, 0; green, 0; blue, 0 }  ,draw opacity=0.63 ]   (518.1,39.42) -- (515.36,60.43) ;
        \draw [shift={(515.1,62.42)}, rotate = 277.43] [color={rgb, 255:red, 0; green, 0; blue, 0 }  ,draw opacity=0.63 ][line width=0.75]    (10.93,-3.29) .. controls (6.95,-1.4) and (3.31,-0.3) .. (0,0) .. controls (3.31,0.3) and (6.95,1.4) .. (10.93,3.29)   ;
        \draw [color={rgb, 255:red, 0; green, 0; blue, 0 }  ,draw opacity=0.63 ]   (543.67,43.88) -- (537.99,64.19) ;
        \draw [shift={(537.45,66.12)}, rotate = 285.62] [color={rgb, 255:red, 0; green, 0; blue, 0 }  ,draw opacity=0.63 ][line width=0.75]    (10.93,-3.29) .. controls (6.95,-1.4) and (3.31,-0.3) .. (0,0) .. controls (3.31,0.3) and (6.95,1.4) .. (10.93,3.29)   ;
        \draw [color={rgb, 255:red, 0; green, 0; blue, 0 }  ,draw opacity=0.63 ]   (566.88,51.29) -- (561.21,71.59) ;
        \draw [shift={(560.67,73.52)}, rotate = 285.62] [color={rgb, 255:red, 0; green, 0; blue, 0 }  ,draw opacity=0.63 ][line width=0.75]    (10.93,-3.29) .. controls (6.95,-1.4) and (3.31,-0.3) .. (0,0) .. controls (3.31,0.3) and (6.95,1.4) .. (10.93,3.29)   ;
        \draw [color={rgb, 255:red, 0; green, 0; blue, 0 }  ,draw opacity=0.63 ]   (589.77,63.28) -- (578.27,80.85) ;
        \draw [shift={(577.17,82.52)}, rotate = 303.22] [color={rgb, 255:red, 0; green, 0; blue, 0 }  ,draw opacity=0.63 ][line width=0.75]    (10.93,-3.29) .. controls (6.95,-1.4) and (3.31,-0.3) .. (0,0) .. controls (3.31,0.3) and (6.95,1.4) .. (10.93,3.29)   ;
        \draw [color={rgb, 255:red, 0; green, 0; blue, 0 }  ,draw opacity=0.63 ]   (612.1,80.3) -- (599.06,96.87) ;
        \draw [shift={(597.83,98.44)}, rotate = 308.19] [color={rgb, 255:red, 0; green, 0; blue, 0 }  ,draw opacity=0.63 ][line width=0.75]    (10.93,-3.29) .. controls (6.95,-1.4) and (3.31,-0.3) .. (0,0) .. controls (3.31,0.3) and (6.95,1.4) .. (10.93,3.29)   ;
        \draw [color={rgb, 255:red, 0; green, 0; blue, 0 }  ,draw opacity=0.63 ]   (629.19,100.99) -- (610.96,111.8) ;
        \draw [shift={(609.24,112.82)}, rotate = 329.33] [color={rgb, 255:red, 0; green, 0; blue, 0 }  ,draw opacity=0.63 ][line width=0.75]    (10.93,-3.29) .. controls (6.95,-1.4) and (3.31,-0.3) .. (0,0) .. controls (3.31,0.3) and (6.95,1.4) .. (10.93,3.29)   ;
        \draw [color={rgb, 255:red, 0; green, 0; blue, 0 }  ,draw opacity=0.63 ]   (580.3,199.93) -- (567.89,182.89) ;
        \draw [shift={(566.71,181.27)}, rotate = 53.93] [color={rgb, 255:red, 0; green, 0; blue, 0 }  ,draw opacity=0.63 ][line width=0.75]    (10.93,-3.29) .. controls (6.95,-1.4) and (3.31,-0.3) .. (0,0) .. controls (3.31,0.3) and (6.95,1.4) .. (10.93,3.29)   ;
        \draw [color={rgb, 255:red, 0; green, 0; blue, 0 }  ,draw opacity=0.63 ]   (560.89,214.18) -- (548.47,197.14) ;
        \draw [shift={(547.29,195.52)}, rotate = 53.93] [color={rgb, 255:red, 0; green, 0; blue, 0 }  ,draw opacity=0.63 ][line width=0.75]    (10.93,-3.29) .. controls (6.95,-1.4) and (3.31,-0.3) .. (0,0) .. controls (3.31,0.3) and (6.95,1.4) .. (10.93,3.29)   ;
        \draw [color={rgb, 255:red, 0; green, 0; blue, 0 }  ,draw opacity=0.63 ]   (632.78,141.2) -- (613.25,133.48) ;
        \draw [shift={(611.39,132.75)}, rotate = 21.57] [color={rgb, 255:red, 0; green, 0; blue, 0 }  ,draw opacity=0.63 ][line width=0.75]    (10.93,-3.29) .. controls (6.95,-1.4) and (3.31,-0.3) .. (0,0) .. controls (3.31,0.3) and (6.95,1.4) .. (10.93,3.29)   ;
        \draw [color={rgb, 255:red, 0; green, 0; blue, 0 }  ,draw opacity=0.63 ]   (616.77,168.82) -- (601.55,154.22) ;
        \draw [shift={(600.11,152.84)}, rotate = 43.81] [color={rgb, 255:red, 0; green, 0; blue, 0 }  ,draw opacity=0.63 ][line width=0.75]    (10.93,-3.29) .. controls (6.95,-1.4) and (3.31,-0.3) .. (0,0) .. controls (3.31,0.3) and (6.95,1.4) .. (10.93,3.29)   ;
        \draw [color={rgb, 255:red, 0; green, 0; blue, 0 }  ,draw opacity=0.63 ]   (601.26,183.29) -- (586.61,167.98) ;
        \draw [shift={(585.23,166.53)}, rotate = 46.27] [color={rgb, 255:red, 0; green, 0; blue, 0 }  ,draw opacity=0.63 ][line width=0.75]    (10.93,-3.29) .. controls (6.95,-1.4) and (3.31,-0.3) .. (0,0) .. controls (3.31,0.3) and (6.95,1.4) .. (10.93,3.29)   ;
        \draw [color={rgb, 255:red, 0; green, 0; blue, 0 }  ,draw opacity=0.63 ]   (430.47,214.59) .. controls (450.22,222.75) and (468.43,225.26) .. (492.7,219.99) ;
        \draw [shift={(494.58,219.57)}, rotate = 167.01] [color={rgb, 255:red, 0; green, 0; blue, 0 }  ,draw opacity=0.63 ][line width=0.75]    (10.93,-3.29) .. controls (6.95,-1.4) and (3.31,-0.3) .. (0,0) .. controls (3.31,0.3) and (6.95,1.4) .. (10.93,3.29)   ;
        \draw [shift={(428.58,213.8)}, rotate = 23.2] [color={rgb, 255:red, 0; green, 0; blue, 0 }  ,draw opacity=0.63 ][line width=0.75]    (10.93,-3.29) .. controls (6.95,-1.4) and (3.31,-0.3) .. (0,0) .. controls (3.31,0.3) and (6.95,1.4) .. (10.93,3.29)   ;
        \draw [color={rgb, 255:red, 0; green, 0; blue, 0 }  ,draw opacity=0.63 ]   (110.58,140.87) -- (107.52,155.14) ;
        \draw [shift={(107.1,157.1)}, rotate = 282.11] [color={rgb, 255:red, 0; green, 0; blue, 0 }  ,draw opacity=0.63 ][line width=0.75]    (10.93,-3.29) .. controls (6.95,-1.4) and (3.31,-0.3) .. (0,0) .. controls (3.31,0.3) and (6.95,1.4) .. (10.93,3.29)   ;
        \draw [color={rgb, 255:red, 0; green, 0; blue, 0 }  ,draw opacity=0.63 ]   (123.58,134.87) -- (130.64,148.1) ;
        \draw [shift={(131.58,149.87)}, rotate = 241.93] [color={rgb, 255:red, 0; green, 0; blue, 0 }  ,draw opacity=0.63 ][line width=0.75]    (10.93,-3.29) .. controls (6.95,-1.4) and (3.31,-0.3) .. (0,0) .. controls (3.31,0.3) and (6.95,1.4) .. (10.93,3.29)   ;
        \draw [color={rgb, 255:red, 0; green, 0; blue, 0 }  ,draw opacity=0.63 ]   (136.58,81.87) -- (152.17,86.14) ;
        \draw [shift={(154.1,86.67)}, rotate = 195.32] [color={rgb, 255:red, 0; green, 0; blue, 0 }  ,draw opacity=0.63 ][line width=0.75]    (10.93,-3.29) .. controls (6.95,-1.4) and (3.31,-0.3) .. (0,0) .. controls (3.31,0.3) and (6.95,1.4) .. (10.93,3.29)   ;
        \draw [color={rgb, 255:red, 0; green, 0; blue, 0 }  ,draw opacity=0.63 ]   (134.58,100.87) -- (148.93,110.63) ;
        \draw [shift={(150.58,111.75)}, rotate = 214.22] [color={rgb, 255:red, 0; green, 0; blue, 0 }  ,draw opacity=0.63 ][line width=0.75]    (10.93,-3.29) .. controls (6.95,-1.4) and (3.31,-0.3) .. (0,0) .. controls (3.31,0.3) and (6.95,1.4) .. (10.93,3.29)   ;
        \draw [color={rgb, 255:red, 0; green, 0; blue, 0 }  ,draw opacity=0.63 ]   (122.58,111.87) -- (129.64,125.1) ;
        \draw [shift={(130.58,126.87)}, rotate = 241.93] [color={rgb, 255:red, 0; green, 0; blue, 0 }  ,draw opacity=0.63 ][line width=0.75]    (10.93,-3.29) .. controls (6.95,-1.4) and (3.31,-0.3) .. (0,0) .. controls (3.31,0.3) and (6.95,1.4) .. (10.93,3.29)   ;
        \draw [color={rgb, 255:red, 0; green, 0; blue, 0 }  ,draw opacity=0.63 ]   (106.58,116.87) -- (103.52,131.14) ;
        \draw [shift={(103.1,133.1)}, rotate = 282.11] [color={rgb, 255:red, 0; green, 0; blue, 0 }  ,draw opacity=0.63 ][line width=0.75]    (10.93,-3.29) .. controls (6.95,-1.4) and (3.31,-0.3) .. (0,0) .. controls (3.31,0.3) and (6.95,1.4) .. (10.93,3.29)   ;
        \draw [color={rgb, 255:red, 0; green, 0; blue, 0 }  ,draw opacity=0.63 ]   (86.58,98.87) -- (85.29,112.11) ;
        \draw [shift={(85.1,114.1)}, rotate = 275.56] [color={rgb, 255:red, 0; green, 0; blue, 0 }  ,draw opacity=0.63 ][line width=0.75]    (10.93,-3.29) .. controls (6.95,-1.4) and (3.31,-0.3) .. (0,0) .. controls (3.31,0.3) and (6.95,1.4) .. (10.93,3.29)   ;
        \draw [color={rgb, 255:red, 0; green, 0; blue, 0 }  ,draw opacity=0.63 ]   (64.58,98.87) -- (63.29,112.11) ;
        \draw [shift={(63.1,114.1)}, rotate = 275.56] [color={rgb, 255:red, 0; green, 0; blue, 0 }  ,draw opacity=0.63 ][line width=0.75]    (10.93,-3.29) .. controls (6.95,-1.4) and (3.31,-0.3) .. (0,0) .. controls (3.31,0.3) and (6.95,1.4) .. (10.93,3.29)   ;
        \draw [color={rgb, 255:red, 0; green, 0; blue, 0 }  ,draw opacity=0.63 ]   (49.58,91.87) -- (44.82,104.23) ;
        \draw [shift={(44.1,106.1)}, rotate = 291.07] [color={rgb, 255:red, 0; green, 0; blue, 0 }  ,draw opacity=0.63 ][line width=0.75]    (10.93,-3.29) .. controls (6.95,-1.4) and (3.31,-0.3) .. (0,0) .. controls (3.31,0.3) and (6.95,1.4) .. (10.93,3.29)   ;
        \draw [color={rgb, 255:red, 0; green, 0; blue, 0 }  ,draw opacity=0.63 ]   (77.58,82.75) -- (75.79,99.76) ;
        \draw [shift={(75.58,101.75)}, rotate = 276.01] [color={rgb, 255:red, 0; green, 0; blue, 0 }  ,draw opacity=0.63 ][line width=0.75]    (10.93,-3.29) .. controls (6.95,-1.4) and (3.31,-0.3) .. (0,0) .. controls (3.31,0.3) and (6.95,1.4) .. (10.93,3.29)   ;
        \draw [color={rgb, 255:red, 0; green, 0; blue, 0 }  ,draw opacity=0.63 ]   (109.58,79.87) -- (123.93,89.63) ;
        \draw [shift={(125.58,90.75)}, rotate = 214.22] [color={rgb, 255:red, 0; green, 0; blue, 0 }  ,draw opacity=0.63 ][line width=0.75]    (10.93,-3.29) .. controls (6.95,-1.4) and (3.31,-0.3) .. (0,0) .. controls (3.31,0.3) and (6.95,1.4) .. (10.93,3.29)   ;
        \draw [color={rgb, 255:red, 0; green, 0; blue, 0 }  ,draw opacity=0.63 ]   (159.58,184.65) -- (166.31,200.26) ;
        \draw [shift={(167.1,202.1)}, rotate = 246.7] [color={rgb, 255:red, 0; green, 0; blue, 0 }  ,draw opacity=0.63 ][line width=0.75]    (10.93,-3.29) .. controls (6.95,-1.4) and (3.31,-0.3) .. (0,0) .. controls (3.31,0.3) and (6.95,1.4) .. (10.93,3.29)   ;
        
        \draw (242.9,118.62) node [anchor=north west][inner sep=0.75pt]    {$\Omega$};
        \draw (201.96,230.84) node [anchor=north west][inner sep=0.75pt]   [align=left] {Foundation};
        \draw (3,55.4) node [anchor=north west][inner sep=0.75pt]    {$\Gamma_{D}$};
        \draw (99.1,89.07) node [anchor=north west][inner sep=0.75pt]    {$f_{0}{}$};
        \draw (471.9,110.94) node [anchor=north west][inner sep=0.75pt]    {$\Omega$};
        \draw (537.96,228.16) node [anchor=north west][inner sep=0.75pt]   [align=left] {Foundation};
        \draw (636.81,124.35) node [anchor=north west][inner sep=0.75pt]    {$f_{2}$};
        \draw (450.83,227.35) node [anchor=north west][inner sep=0.75pt]    {$f_{3}$};
        \draw (125,230.4) node [anchor=north west][inner sep=0.75pt]    {$\Gamma_{C}$};
        \draw (284,129.4) node [anchor=north west][inner sep=0.75pt]    {$\Gamma_{N}$};
        \draw (374,105.4) node [anchor=north west][inner sep=0.75pt]    {$\Gamma_{D}$};
        \draw (452,185.4) node [anchor=north west][inner sep=0.75pt]    {$\Gamma_{C}$};
        \draw (566,121.4) node [anchor=north west][inner sep=0.75pt]    {$\Gamma_{N}$};
    \end{tikzpicture}
    
    \caption{A viscoelastic material with long memory subjected to a external body force $f_{0}$ (left), normal traction $f_{2}$ and dry friction $f_{3}$ (right).}
    \label{fig:ConfigPhysical}
\end{figure}
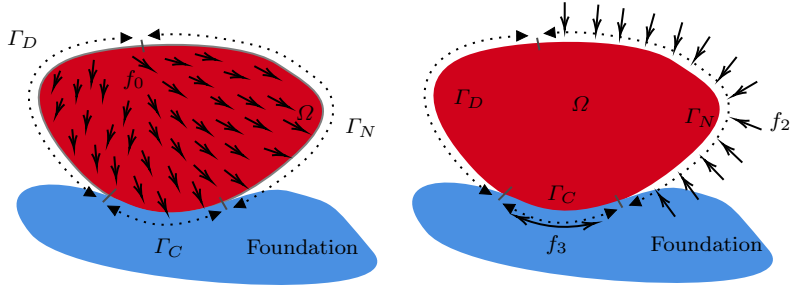

To describe the mathematical model,  let us consider \(\mathbb{S}^{d}\) as the space of symmetric second-order tensors over \(\mathbb{R}^{d}\). On \(\mathbb{R}^{d}\) and \(\mathbb{S}^{d}\), we consider the inner products
$$
    \langle u, v \rangle_{\mathbb{R}^{d}} := \sum_{i=1}^{d} u_{i} v_{i}
    \,\, \text{and} \,\,
    \langle \sigma, \tau \rangle_{\mathbb{S}^{d}} := \sum_{i,j=1}^{d} \sigma_{ij} \tau_{ij} \,\, \text{for all } u, v \in \mathbb{R}^{d},  \sigma, \tau \in \mathbb{S}^{d}.
$$
These inner products induce the norms \(\|u\|_{\mathbb{R}^{d}} := \sqrt{\langle u, u \rangle_{\mathbb{R}^{d}}}\) and \(\|\sigma\|_{\mathbb{S}^{d}} := \sqrt{\langle \sigma, \sigma \rangle_{\mathbb{S}^{d}}}\) on \(\mathbb{R}^{d}\) and \(\mathbb{S}^{d}\), respectively. For simplicity, we will denote these norms by \(\|\cdot\|\) whenever no ambiguity arises.

Given $u\in (H^{1}(\Omega))^{d}$ we denote by $\varepsilon(u)$ the \emph{strain operator} defined by $\varepsilon(u) := (\varepsilon_{ij}(u))_{i,j=1,\dots,d}\in \mathbb{S}^{d}$, where 
\begin{eqnarray*}
   \varepsilon_{ij}(u) := \frac{1}{2}\left(\frac{\partial u_{i}}{\partial x_{j}} + \frac{\partial u_{j}}{\partial x_{i}} \right).
\end{eqnarray*}
To model the situation depicted in Fig. \ref{fig:ConfigPhysical}, we  consider the vector spaces:
\begin{align*}
    V := \{u\in(H^{1}(\Omega))^{d} : u=0 \textrm{ on }\Gamma_{D}\} \textrm{ and } W := \{\sigma \in \mathbb{S}^{d}  : \sigma_{ij} = \sigma_{ji} \in L^{2}(\Omega)\},
\end{align*}
where the equality $u=0$ on $\Gamma_{D}$ is understood in the sense of traces. The above spaces are endowed with the inner products:
\begin{eqnarray*}
    \langle u,v \rangle_{V} := \int_{\Omega} \varepsilon(u(x))\cdot\varepsilon(v(x)) dx \quad\textrm{ and }\quad \langle \sigma,\tau \rangle_{W} := \int_{\Omega} \sigma(x) \cdot \tau(x) dx.
\end{eqnarray*}

We denote by \(\nu\) the unit outward vector to \(\Gamma\). The normal and tangential components of the vector \(u \in \mathbb{R}^{d}\) are \(u_{\nu}:=u \cdot \nu \) and \(u_{\tau}:=u - u_{\nu} \nu\), respectively. Similarly, \(\sigma_{\nu}:=(\sigma \nu) \cdot \nu\) and \(\sigma_{\tau}:=\sigma \nu - \sigma_{\nu} \nu\) are the normal and tangential components of the tensor \(\sigma \in \mathbb{S}^{d}\), respectively.

\paragraph{Description of the mechanical model:} The considered constitutive law has the structure of a general viscoelastic law with long memory, including a non-linear integral part (see \cite{MR2976197}), given by 
\begin{eqnarray*}
    \sigma(t,x) :=
    \mathcal{A}(x,\varepsilon(u(t,x)))
    +
    \int_{0}^{t} \mathcal{B}(t-s,x,\varepsilon(u(t,x)))\,ds,
\end{eqnarray*}
for all $(t,x)\in I\times \Omega$, where $\sigma(t,x)$ denotes the stress tensor, $u(t,x)$ is the displacement field and $\mathcal{A}$ and $\mathcal{B}$ denote the \emph{elasticity} and the \emph{relaxation} operators, respectively.  The above constitutive law describes the stress response of a material that exhibits both immediate elastic behavior and time-dependent memory effects, which are present, for example, in polymers, biological tissues, and certain metals. It provides a framework for predicting phenomena such as creep and stress relaxation. For further details on applications and notation, we refer to \cite{MR2976197,MR4403784,MR3908332,MR3752610}.

The problem can be formulated as follows. 
\begin{problem}[Contact problem]
    \label{problema-mecanico}
    Find a displacement vector field $u\colon I\times \Omega \to \mathbb{R}^d$ such that
    \begin{align}
        \operatorname{div}(\sigma(t,x))+f_0(t,x)&=0 \quad\quad\quad\quad  \textrm{ for all } (t,x)\in I\times \Omega, \label{p.1}\\
        u(t,x)&=0 \quad\quad \quad\quad \textrm{ for all } (t,x)\in I\times \Gamma_{D}, \label{p.2}\\
        \sigma(t,x)\nu(x)&=f_2(t,x)  \,\,\,\,\,\,\,\,\, \textrm{ for all } (t,x)\in I\times \Gamma_{N}, \label{p.3}\\
        u_{\nu}(t,x)&=0 \quad\quad \quad\quad\textrm{ for all } (t,x)\in I\times \Gamma_{C}. \label{p.4}
    \end{align}
    Moreover, dry friction is acting on $\Gamma_{C}$, i.e., for all $(t,x)\in I\times \Gamma_{C}$ 
    \begin{equation}
        \label{dry-friction}
        \Vert \sigma_{\tau}(t,x)\Vert \leq f_{3}(t,x) \textrm{ and } -\sigma_{\tau}(t,x)=f_3(t,x)\frac{\dot{u}_{\tau}(t,x)}{\Vert \dot{u}_{\tau}(t,x)\Vert} \textrm{ if }  \dot{u}_{\tau}(t,x)\neq 0.
    \end{equation}
\end{problem}
Equation \eqref{p.1} in Problem \ref{problema-mecanico} represents the equilibrium balance between the stress tensor and an external force field  $f_{0}(t,x)$, while \eqref{p.2} indicates that the set $\Omega$ is clamped on $\Gamma_{D}$. Moreover,  \eqref{p.3} states that a normal traction force $f_{2}(t,x)$ is applied on  $\Gamma_{N}$. Equation \eqref{p.4} specifies that no normal displacements are allowed in the contact region $\Gamma_C$. Finally, \eqref{dry-friction} represents the dry friction, defined by a friction bound $f_{3}(t,x)$.

\paragraph{Assumptions on the data:} In order to provide the well-posedness for the Problem \ref{problema-mecanico}, we consider the following hypotheses:

\noindent $(\mathcal{H}^{\Gamma})$: Assumptions on $f_{0}$, $f_{2}$ and $f_{3}$:
\begin{itemize}
    \item[(i)] There exists an absolutely continuous function $\vartheta_0$ such that
    $$
    \Vert f_0(t,\cdot)-f_0(s,\cdot)\Vert_{(L^{2}(\Omega))^{d}} \leq \vert \vartheta_0(t)-\vartheta_0(s)\vert \textrm{ for all } t,s\in I.
    $$
    \item[(ii)] There exists an absolutely continuous function $\vartheta_2$ such that
    $$
    \Vert f_2(t,\cdot)-f_2(s,\cdot)\Vert_{(L^{2}(\Omega))^{d}} \leq \vert \vartheta_2(t)-\vartheta_2(s)\vert \textrm{ for all } t,s\in I.
    $$
    \item[(iii)] For all $t\in I$, the map $x\to f_3(t,x)$ is measurable  and $f_3(0,x)=0$ for a.e. $x\in \Gamma_{C}$.
\end{itemize}

\noindent $(\mathcal{H}^{\sigma})$: Assumptions on $\mathcal{A}$ and $\mathcal{B}$:
\begin{enumerate}
    \item[(i)] The \textit{elasticity operator} $\mathcal{A}=(a_{ijkl})$ defined from $\Omega\times\mathbb{S}^{d}$ into $\mathbb{S}^{d}$, is a linear application respect his second variable, such that satisfies:\\
    \noindent $(a)$ The components are symmetric, i.e., $a_{ijkl}=a_{jikl}=a_{klij}$, and $a_{ijkl}$ is essentially bounded, i.e., $a_{ijkl}\in L^{\infty}(\Omega\times\mathbb{S}^{d})$\\
    \noindent $(b)$ There exists $m_{\mathcal{A}}>0$ such that
    \begin{eqnarray*}
        \langle \mathcal{A}(x,\varepsilon),\varepsilon\rangle_{\mathbb{S}^{d}} \geq m_{\mathcal{A}}\|\varepsilon\|_{\mathbb{S}^{d}}^{2} \quad\textrm{ for all }\varepsilon\in\mathbb{S}^{d}.
    \end{eqnarray*}
    \item[(ii)] The \textit{relaxation operator} $\mathcal{B}=(b_{ijkl})$ defined from $I\times\Omega\times\mathbb{S}^{d}$ into $\mathbb{S}^{d}$ satisfies the following conditions:\\
    \noindent $(a)$ The components are symmetric, i.e., $b_{ijkl} = b_{jikl} = b_{klij}$. \\
    \noindent $(b)$ For all $\varepsilon\in\mathbb{S}^{d}$ the map $t\mapsto b_{ijkl}(t,\cdot,\varepsilon)$ belongs to $W^{1,2}(I;L^{\infty}(\Omega))$.\\
    \noindent $(c)$ There exist nonnegative constants  $\kappa_{\mathcal{B}}$ and $\kappa_{\mathcal{B}^{\prime}}$ such that
    \begin{itemize}
        \item For all $(t,x)\in I\times\Omega$ and any $\varepsilon_{1},\varepsilon_{2}\in\mathbb{S}^{d}$ one has 
        \begin{align*}
            &\|\mathcal{B}(t,x,\varepsilon_{1}) - \mathcal{B}(t,x,\varepsilon_{2})\|_{\mathbb{S}^{d}} \leq \kappa_{\mathcal{B}} \|\varepsilon_{1} - \varepsilon_{2}\|_{\mathbb{S}^{d}}.
        \end{align*}
        \item For a.e. $t\in I$, for all $x\in\Omega$ and  any $\varepsilon_{1},\varepsilon_{2}\in\mathbb{S}^{d}$ one has
        \begin{align*}
            \|\mathcal{B}'(t,x,\varepsilon_{1}) - \mathcal{B}'(t,x,\varepsilon_{2})\|_{\mathbb{S}^{d}} \leq \kappa_{\mathcal{B}^{\prime}} \|\varepsilon_{1} - \varepsilon_{2}\|_{\mathbb{S}^{d}}.
        \end{align*}
    \end{itemize}
    \noindent$(d)$ For any $(t,x)\in I\times\Omega$ one has
    \begin{eqnarray*}
        \langle \mathcal{B}(t,x,\varepsilon_{1}) - \mathcal{B}(t,x,\varepsilon_{2}) , \varepsilon_{1}-\varepsilon_{2} \rangle_{\mathbb{S}^{d}} \geq 0 \quad\textrm{for all } \varepsilon_{1},\varepsilon_{2}\in\mathbb{S}^{d}.
    \end{eqnarray*}
\end{enumerate}
Now, based on standard calculations, the weak formulation of \ref{problema-mecanico} is given by:
\begin{problem}[Variational inequality]
    \label{problem_6.2}
    Find a displacement field $u(t,x)$ such that the following inequality holds: for all $v\in V$,
    \begin{align*} 
        \int_{\Omega} \sigma(t,x) : ( \varepsilon(v(x)) - \varepsilon(\dot{u}(t,x)) ) dx + \int_{\Gamma_{C}} f_{3}(t,x) (\|v_{\tau}(x)\| - \|\dot{u}_{\tau}(t,x)\|) da(x)\\
        \quad\quad\quad\quad\quad \geq 
        \int_{\Omega} f_{0}(t,x) (v(x) - \dot{u}(t,x)) dx 
        + \int_{\Gamma_{N}} f_{2}(t,x) ( v(x) - \dot{u}(t,x) )da(x).
    \end{align*}
\end{problem}
To apply our theoretical result (Theorem \ref{teo:Existence1}), we introduce the following auxiliary functions.
\begin{definition}[Auxiliary functions]
    \label{def:AuxiliaryFunctions}
    Let $f_{0}, f_{2}, f_{3}, \mathcal{A}$ and $\mathcal{B}$ satisfying the assumptions $(\mathcal{H}^{\Gamma})$ and $(\mathcal{H}^{\sigma})$. We consider the following functions: 
    \begin{enumerate}
        \item Let $j\colon V\to\R$ defined by
        $$
            j(u) = \int_{\Gamma_{C}} f_{3}(t,x) \|u_{\tau}(x)\| da(x) \quad\textrm{ for all }t\in I.
        $$
        \item Let $f \colon I \to  V$ defined by 
        $$ 
            \langle f(t) , u \rangle_{V} := \int_{\Omega} f_{0}(t,x) \cdot u(x)dx + \int_{\Gamma_{N}} f_{2}(t,x) \cdot u(x) da(x) \quad \textrm{ for all } t\in I.
        $$  
        \item Let $A\colon V\to V$ and $B\colon C(I;V) \to C(I;V)$ be such that:
        \begin{align*}
            \langle A(u),v \rangle _{V}   &= \int_{\Omega} \mathcal{A}(x,\varepsilon(u(x))) \cdot \varepsilon(v(x)) dx \\
            \langle B(u)(t),v \rangle_{V} &= \int_{\Omega} \left( \int_{0}^{t} \mathcal{B}(t-s,x, \varepsilon(u(s,x)))ds \right) \cdot \varepsilon(v(x)) dx.
        \end{align*}
        Furthermore, we set $P\colon V\to V$ be such that $A = P^{\ast} P$, where $P^{\ast}$ denotes the adjoint operator of $P$ and $Q:=(P^{\ast})^{-1}$.

        \item Let $C\colon I \rightrightarrows V$ be a set-valued map defined by $C(t) := Q(f(t) - \partial j(0))$.

        \item Let $R\colon C(I;V)\to C(I;V)$ be defined by
    \end{enumerate}
    \begin{align*}
        \langle R(u)(t) , v\rangle_{V} 
        =& \int_{\Omega}  \mathcal{B}(0,x,\varepsilon(u(t,x)))\cdot  \varepsilon(v(x)) dx \\
        &\quad \quad \quad + \int_{\Omega} \left(\int_{0}^{t} \mathcal{B}'(t-s,x,\varepsilon(u(s,x)))ds \right) \cdot  \varepsilon(v(x)) dx.
    \end{align*}
\end{definition}
We now present the main result of this section, which establishes the existence and uniqueness of a weak solution to Problem \ref{problem_6.2}. Furthermore, we demonstrate that this solution satisfies a related history-dependent sweeping process.
\begin{theorem}
    Assume that $(\mathcal{H}^{\Gamma})$ and $(\mathcal{H}^{\sigma})$ hold. Then, the Problem \eqref{problem_6.2} admits a unique absolutely continuous solution $u\colon I \to V$ to the differential inclusion:
    \begin{equation*}
    \left\{
    \begin{aligned}
        \dot{u}(t) & \in -N_{C(t)}(u(t)) + QR(u)(t) \quad\textrm{ a.e. }t\in I,\\
        u(0)       & = Pu_{0},
    \end{aligned}
    \right.
    \end{equation*}
    where $Q=(P^{\ast})^{-1}$ and $A=P^{\ast}P$.
\end{theorem}

\begin{proof}
    According to Definition \ref{def:AuxiliaryFunctions}, the Problem \ref{problem_6.2} is equivalent to the following variational inequality: find $u\colon I\to V$ such that 
    \begin{align*}
        j(\dot{u}(t)) \geq j(v) + \langle f(t) - A(u(t)) - B(u)(t) , \dot{u}(t) - v  \rangle_{V} \textrm{ for all } (t,v)\in I\times V.
    \end{align*}
    Since the function $j\colon V\to \mathbb{R}$ is convex, the above inequality is equivalent to: 
    \begin{eqnarray*}
        f(t) - A(u(t)) - B(u)(t) \in \partial j( \dot{u}(t) ) \quad\textrm{ for all }t \in I.
    \end{eqnarray*}
    Moreover, we observe that $j$ is positively homogeneous, which implies that  $$\partial j(\dot{u}(t)) = \partial \sigma_{\partial j(0)}(\dot{u}(t)) \textrm{ for a.e.  } t\in I.$$
    Therefore, for a.e. $t\in I$, one has
    \begin{equation*}
        \begin{aligned}
            \dot{u}(t) & \in
             \partial \sigma_{\partial j(0)}^{*} (f(t) - A(u(t)) - B(u)(t))  \\
            &= N_{\partial j(0)} (f(t) - A(u(t)) - B(u)(t))\\
            &= N_{\partial j(0)-f(t)}(-P^{\ast} P(u(t)) - B(u)(t))\\
            &=-N_{f(t)-\partial j(0)}(P^{\ast} P(u(t)) + B(u)(t)).
        \end{aligned}
    \end{equation*}
    Following the ideas of \cite{MR3912745}, we consider the operator $\mathcal{G}\colon V \to V$ defined by 
    $$\mathcal{G}(u)(t) := Pu(t) + QB(u)(t).$$
    It is clear that $\mathcal{G}$ is invertible with Lipschitz inverse (see, e.g., \cite[Proposition 31.4]{MR1033498}). Therefore, setting  $w(t):=\mathcal{G}(u)(t)$, we obtain that which gives
    $$
        \dot{w}(t) \in  -N_{C(t)}(w(t)) + QR(\mathcal{G}^{-1}(w))(t) \textrm{ a.e. } t\in I.
    $$
    It can be noted that $QR\mathcal{G}^{-1}$, and $C$ satisfy the hypotheses $(\mathcal{H}^{R})$ and $ (\mathcal{H}^{C})$, respectively. Therefore, by virtue of Theorem \ref{teo:Existence1}, there exists a unique solution $u\in\operatorname{AC}(I;V)$ for the history-dependent sweeping process.
 \end{proof}

 \begin{acknowledgements}
The third author was supported by ANID Chile under grants Fondecyt Regular  N$^{\circ}$ 1240120, Fondecyt Regular N$^{\circ}$ 1220886,  Proyecto de Exploraci\'on 13220097,  CMM BASAL funds for Center of Excellence FB210005, Project ECOS230027, MATH-AMSUD 23-MATH-17).
\end{acknowledgements}


\begin{thebibliography}{10}

\bibitem{MR3753582}
S.~Adly.
\newblock {\em A variational approach to nonsmooth dynamics}.
\newblock SpringerBriefs Math. Springer, Cham, 2017.

\bibitem{MR3908332}
S.~Adly and M.~Sofonea.
\newblock Time-dependent inclusions and sweeping processes in contact
  mechanics.
\newblock {\em Z. Angew. Math. Phys.}, 70(2):Paper No. 39, 19 pages, 2019.

\bibitem{MR2378491}
C.~D. Aliprantis and K.~C. Border.
\newblock {\em Infinite dimensional analysis}.
\newblock Springer, Berlin, third edition, 2006.

\bibitem{MR4422386}
A.~Bouach, T.~Haddad, and B.~S. Mordukhovich.
\newblock Optimal control of nonconvex integro-differential sweeping processes.
\newblock {\em J. Differential Equations}, 329:255--317, 2022.

\bibitem{MR4492538}
A.~Bouach, T.~Haddad, and L.~Thibault.
\newblock Nonconvex integro-differential sweeping process with applications.
\newblock {\em SIAM J. Control Optim.}, 60(5):2971--2995, 2022.

\bibitem{MR2159846}
M.~Bounkhel and L.~Thibault.
\newblock Nonconvex sweeping process and prox-regularity in {H}ilbert space.
\newblock {\em J. Nonlinear Convex Anal.}, 6(2):359--374, 2005.

\bibitem{MR2328857}
N.~Chemetov and M.~D.~P. Monteiro~Marques.
\newblock Non-convex quasi-variational differential inclusions.
\newblock {\em Set-Valued Anal.}, 15(3):209--221, 2007.

\bibitem{christensen2012theory}
R.~Christensen.
\newblock {\em Theory of Viscoelasticity: An Introduction}.
\newblock Academic Press, 2012.

\bibitem{MR1488695}
F.~H. Clarke, Yu.~S. Ledyaev, R.~J. Stern, and P.~R. Wolenski.
\newblock {\em Nonsmooth analysis and control theory}, volume 178 of {\em
  Graduate Texts in Mathematics}.
\newblock Springer-Verlag, New York, 1998.

\bibitem{MR4099068}
G.~Colombo and C.~Kozaily.
\newblock Existence and uniqueness of solutions for an integral perturbation of
  {M}oreau's sweeping process.
\newblock {\em J. Convex Anal.}, 27(1):229--238, 2020.

\bibitem{MR2768810}
G.~Colombo and L.~Thibault.
\newblock Prox-regular sets and applications.
\newblock In {\em Handbook of nonconvex analysis and applications}, pages
  99--182. Int. Press, Somerville, MA, 2010.

\bibitem{MR0787404}
K.~Deimling.
\newblock {\em Nonlinear functional analysis}.
\newblock Springer-Verlag, Berlin, 1985.

\bibitem{MR1189795}
K.~Deimling.
\newblock {\em Multivalued differential equations}, volume~1 of {\em De Gruyter
  Series in Nonlinear Analysis and Applications}.
\newblock Walter de Gruyter \& Co., Berlin, 1992.

\bibitem{MR2179241}
J.~F. Edmond and L.~Thibault.
\newblock Relaxation of an optimal control problem involving a perturbed
  sweeping process.
\newblock {\em Math. Program.}, 104(2-3):347--373, 2005.

\bibitem{MR110078}
H.~Federer.
\newblock Curvature measures.
\newblock {\em Trans. Amer. Math. Soc.}, 93:418--491, 1959.

\bibitem{MR3626639}
A.~Jourani and E.~Vilches.
\newblock Moreau-{Y}osida regularization of state-dependent sweeping processes
  with nonregular sets.
\newblock {\em J. Optim. Theory Appl.}, 173(1):91--116, 2017.

\bibitem{MR3912745}
A.~Jourani and E.~Vilches.
\newblock A differential equation approach to implicit sweeping processes.
\newblock {\em J. Differential Equations}, 266(8):5168--5184, 2019.

\bibitem{lakes2009viscoelastic}
R.S. Lakes.
\newblock {\em Viscoelastic Materials}.
\newblock Cambridge University Press, 2009.

\bibitem{MR2976197}
S.~Mig\'{o}rski, A.~Ochal, and M.~Sofonea.
\newblock {\em Nonlinear inclusions and hemivariational inequalities},
  volume~26 of {\em Adv. Mech. Math.}
\newblock Springer, New York, 2013.

\bibitem{MR0637727}
J.-J. Moreau.
\newblock Rafle par un convexe variable. {I}.
\newblock In {\em Travaux du {S}\'{e}minaire d'{A}nalyse {C}onvexe, {V}ol.
  {I}}, volume No. 118 of {\em Secr\'{e}tariat des Math\'{e}matiques,
  Publication}, pages Exp. No. 15, 43. Univ. Sci. Tech. Languedoc, Montpellier,
  1971.

\bibitem{MR0637728}
J.-J. Moreau.
\newblock Rafle par un convexe variable. {II}.
\newblock In {\em Travaux du {S}\'{e}minaire d'{A}nalyse {C}onvexe, {V}ol.
  {II}}, volume No. 122 of {\em Secr\'{e}tariat des Math\'{e}matiques,
  Publication}, pages Exp. No. 3, 36. Univ. Sci. Tech. Languedoc, Montpellier,
  1972.

\bibitem{MR0508661}
J.-J. Moreau.
\newblock Evolution problem associated with a moving convex set in a hilbert
  space.
\newblock {\em J. Differential Equations}, 26(3):347--374, 1977.

\bibitem{MR4382708}
F.~Nacry and M.~Sofonea.
\newblock History-dependent operators and prox-regular sweeping processes.
\newblock {\em Fixed Point Theory Algorithms Sci. Eng.}, pages Paper No. 5, 23,
  2022.

\bibitem{MR4403784}
F.~Nacry and M.~Sofonea.
\newblock History-dependent sweeping processes in contact mechanics.
\newblock {\em J. Convex Anal.}, 29(1):77--100, 2022.

\bibitem{MR4421900}
D.~Narv\'{a}ez and E.~Vilches.
\newblock Moreau-{Y}osida regularization of degenerate state-dependent sweeping
  processes.
\newblock {\em J. Optim. Theory Appl.}, 193(1-3):910--930, 2022.

\bibitem{MR1694378}
R.~A. Poliquin, R.~T. Rockafellar, and L.~Thibault.
\newblock Local differentiability of distance functions.
\newblock {\em Trans. Amer. Math. Soc.}, 352(11):5231--5249, 2000.

\bibitem{MR3222899}
M.~Sene and L.~Thibault.
\newblock Regularization of dynamical systems associated with prox-regular
  moving sets.
\newblock {\em J. Nonlinear Convex Anal.}, 15(4):647--663, 2014.

\bibitem{MR3752610}
M.~Sofonea and S.~Mig\'{o}rski.
\newblock {\em Variational-hemivariational inequalities with applications}.
\newblock Monogr. Res. Notes Math. CRC Press, Boca Raton, FL, 2018.

\bibitem{MR1994056}
L.~Thibault.
\newblock Sweeping process with regular and nonregular sets.
\newblock {\em J. Differential Equations}, 193(1):1--26, 2003.

\bibitem{MR2399209}
L.~Thibault.
\newblock Regularization of nonconvex sweeping process in {H}ilbert space.
\newblock {\em Set-Valued Anal.}, 16(2-3):319--333, 2008.

\bibitem{MR4659163}
L.~Thibault.
\newblock {\em Unilateral Variational Analysis in Banach spaces. Part II:
  Special Classes of Functions and Sets}.
\newblock World Scientific Publishing Co. Pte. Ltd., 2023.

\bibitem{MR3813128}
E.~Vilches.
\newblock Regularization of perturbed state-dependent sweeping processes with
  nonregular sets.
\newblock {\em J. Nonlinear Convex Anal.}, 19(4):633--651, 2018.

\bibitem{Vilches-2024}
E.~Vilches.
\newblock Well-posedness for integro-differential sweeping processes of
  {V}olterra type.
\newblock {\em J. Convex Anal.}, 31(4):1273--1288, 2024.

\bibitem{MR1033498}
E.~Zeidler.
\newblock {\em Nonlinear functional analysis and its applications. {II}/{B}}.
\newblock Springer-Verlag, New York, 1990.
\newblock Nonlinear monotone operators, Translated from the German by the
  author and Leo F. Boron.

\end{thebibliography}

\end{document}